\documentclass[12pt]{article}
\usepackage{latexsym,amssymb,upref,amsmath,amsthm, amsfonts,authblk}
\usepackage{color}
\usepackage{tikz}
\usepackage{fullpage}

\usepackage{ucs}
\usepackage{amssymb}
\usepackage{amsthm}
\usepackage{amsmath}
\usepackage{latexsym}
\usepackage[cp1251]{inputenc}
\usepackage[english]{babel}
\usepackage{graphicx}
\usepackage{wrapfig}
\usepackage{caption}
\usepackage{subcaption}
\usepackage{txfonts}
\usepackage{lmodern}
\usepackage{mathrsfs}
\usepackage{algorithm}
\usepackage{dsfont}
\usepackage{enumerate}
\usepackage{multicol}
\usepackage{csquotes}
\usepackage{stmaryrd}
\usepackage{tikz}
\usepackage{comment}
\usepackage{enumitem}
\usepackage{longtable}
\usepackage{cite}

\usepackage[pdftex,hypertexnames=false,linktocpage=true]{hyperref}

\newcommand{\vc}[1]{\ensuremath{\vcenter{\hbox{#1}}}}

\tikzset{unlabeled_vertex/.style={inner sep=1.7pt, outer sep=0pt, circle, fill}} 
\tikzset{labeled_vertex/.style={inner sep=2.2pt, outer sep=0pt, rectangle, fill=yellow, draw=black}} 
\tikzset{edge_color0/.style={color=black,line width=1.2pt,opacity=0.5}} 
\tikzset{edge_color1/.style={color=red,  line width=1.2pt,opacity=1}} 
\tikzset{edge_color2/.style={color=blue, line width=1.2pt,opacity=1}} 
\tikzset{edge_color3/.style={color=green,line width=1.2pt}} 
\tikzset{edge_color4/.style={color=red,  line width=1.2pt,dotted}} 
\tikzset{edge_color5/.style={color=blue, line width=1.2pt,dotted}} 
\tikzset{edge_color6/.style={color=green, line width=1.2pt,dotted}} 
\tikzset{edge_color7/.style={color=orange, line width=1.2pt}} 
\tikzset{edge_color8/.style={color=gray, line width=1.2pt}} 
\tikzset{edge_thin/.style={color=black}} 
\tikzset{edge_hidden/.style={color=black,dotted,opacity=0}} 
\tikzset{vertex_color1/.style={inner sep=1.7pt, outer sep=0pt, draw, circle, fill=red}} 
\tikzset{vertex_color2/.style={inner sep=1.7pt, outer sep=0pt, draw, circle, fill=blue}} 
\tikzset{vertex_color3/.style={inner sep=1.7pt, outer sep=0pt, draw, circle, fill=green}} 
\tikzset{labeled_vertex_color1/.style={inner sep=2.2pt, outer sep=0pt, draw, rectangle, fill=red}} 
\tikzset{labeled_vertex_color2/.style={inner sep=2.2pt, outer sep=0pt, draw, rectangle, fill=blue}} 
\tikzset{labeled_vertex_color3/.style={inner sep=2.2pt, outer sep=0pt, draw, rectangle, fill=green}} 
\def\outercycle#1#2{ \draw \foreach \x in {0,1,...,#2}{(0.5*\x,0) coordinate(x\x)}; 
\path (0,0.3) -- (1,0.3); 
} 
\def\drawhypervertex#1#2{ \draw[edge_color2] (x#1)++(0,-0.2-0.2*#2)+(-0.2,0) -- +(0.2,0);} 
 
\def\drawhyperedge#1#2{ \draw[dotted] (x0)++(0,-0.2-0.2*#1)--++(0.5*#2-0.5,0);
\path (0,-0.4-0.2*#1) -- (0,0); 
} 

\tikzset{
vtx/.style={inner sep=1.1pt, outer sep=0pt, circle, fill,draw}, 
vtxl/.style={inner sep=1.1pt, outer sep=0pt, rectangle, fill=yellow,draw=black}, 
hyperedge/.style={fill=pink,opacity=0.5,draw=black}, 
}

\newcommand{\Kfourthree}{\vc{
\begin{tikzpicture}
\draw
(0,0) coordinate(1) node[vtx](a){}
(1,0) coordinate(2) node[vtx](b){}
(1,1) coordinate(3) node[vtx](c){}
(0,1) coordinate(4) node[vtx](d){}
;
\draw[hyperedge] (1) to[out=30,in=90] (2) to[out=100,in=260] (3) to[out=260,in=30] (1);
\draw[hyperedge] (1) to[out=90,in=150] (2) to[out=150,in=280] (4) to[out=280,in=90] (1);
\draw[hyperedge] (1) to[out=60,in=200] (3) to[out=200,in=330] (4) to[out=330,in=60,looseness=1.5] (1);
\draw[hyperedge] (2) to[out=120,in=230,,looseness=1.8] (3) to[out=230,in=300,looseness=1.8] (4) to[out=300,in=120,looseness=1.8] (2);
\end{tikzpicture}
}}

\newcommand{\Kfourthreeminus}{\vc{
\begin{tikzpicture}
\draw
(0,0) coordinate(1) node[vtx](a){}
(1,0) coordinate(2) node[vtx](b){}
(1,1) coordinate(3) node[vtx](c){}
(0,1) coordinate(4) node[vtx](d){}
;
\draw[hyperedge] (1) to[out=30,in=90] (2) to[out=100,in=260] (3) to[out=260,in=30] (1);
\draw[hyperedge] (1) to[out=90,in=150] (2) to[out=150,in=280] (4) to[out=280,in=90] (1);
\draw[hyperedge] (1) to[out=60,in=200] (3) to[out=200,in=330] (4) to[out=330,in=60,looseness=1.5] (1);
\end{tikzpicture}
}}

\newtheorem{thm}{Theorem} 
\newtheorem{lemma}[thm]{Lemma}
\newtheorem{cor}[thm]{Corollary}
\theoremstyle{definition}
\newtheorem{defn}[thm]{Definition}

\theoremstyle{definition}
\newtheorem{const}{Construction}



\newcommand{\F}{\mathcal{F}}

\newcommand{\floor}[1]{\left\lfloor{#1}\right\rfloor}

\DeclareMathOperator{\ex}{ex}
\DeclareMathOperator{\coex}{coex}

\title{Positive co-degree density of hypergraphs}

\author{Anastasia Halfpap\thanks{Department of Mathematical Sciences, University of Montana. Email: \texttt{anastasia.halfpap@umontana.edu}.}
\qquad
Nathan Lemons\thanks{Theoretical Division, Los Alamos National Laboratory, Email: \texttt{nlemons@lanl.gov}.}
\qquad
Cory Palmer\thanks{Department of Mathematical Sciences, University of Montana. Email: \texttt{cory.palmer@umontana.edu}.
Research supported by a grant from the Simons Foundation \#712036.}}

\begin{document}

\maketitle

\begin{abstract}
The \emph{minimum positive co-degree} of a non-empty $r$-graph ${H}$, denoted $\delta_{r-1}^+( {H})$, is the maximum $k$ such that if $S$ is an $(r-1)$-set contained in a hyperedge of $ {H}$, then $S$ is contained in at least $k$ distinct hyperedges of $ {H}$. Given an $r$-graph ${F}$, we introduce the {\it positive co-degree Tur\'an number} $\mathrm{co^+ex}(n, {F})$ as the maximum positive co-degree $\delta_{r-1}^+(H)$ over all $n$-vertex $r$-graphs $H$ that do not contain $F$ as a subhypergraph.

In this paper we concentrate on the behavior of $\mathrm{co^+ex}(n, {F})$ for $3$-graphs $F$. In particular, we determine asymptotics and bounds for several well-known concrete $3$-graphs $F$ (e.g.\ $K_4^-$ and the Fano plane). 
We also show that, for $r$-graphs, the limit
\[
\gamma^+(F) := \lim_{n \rightarrow \infty} \frac{\mathrm{co^+ex}(n, {F})}{n}
\]
exists, and ``jumps'' from $0$ to $1/r$, i.e., it never takes on values in the interval $(0,1/r)$. Moreover, we characterize which $r$-graphs $F$ have $\gamma^+(F)=0$.
Our motivation comes primarily from the study of (ordinary) co-degree Tur\'an numbers where a number of results have been proved that inspire our results.
\end{abstract}

\section{Introduction}

An {\it $r$-graph} is a hypergraph whose edges (called {\it$r$-edges}) all have size $r$, i.e., an $r$-uniform hypergraph. Given a hypergraph $H$, we use $V(H)$, or simply $V$, to denote the vertex set of $H$ and $E(H)$ or $E$ for the edge set. We often denote a hyperedge by the concatenation of its vertices. For a family of hypergraphs $\F$, a hypergraph is {\it $\F$-free} if it does not contain a member of $\F$ as a subhypergraph\footnote{ 
 When the forbidden family $\F$ consists of a single hypergraph $F$ we will write $F$ in place of $\F=\{F\}$.}.
The {\it Tur\'an number} (or {\it extremal number}) $\ex_r(n,\F)$ is the maximum number of edges in an $n$-vertex $\F$-free $r$-graph. The {\it Tur\'an density} of $\F$ is the limit
\[
\pi(\F) := \lim_{n \rightarrow \infty} \frac{\ex_r(n,\F)}{\binom{n}{r}}
\]
which exists by an argument due to Katona, Nemetz and Simonovits\cite{KNS}. For graphs (i.e., when $r=2$), much is known about the Tur\'an density. For example, the Erd\H os-Stone theorem~\cite{ErdosSimonovits,ErdosStone} determines $\pi(F)$. On the other hand, when $r\geq 3$, very few exact results are known. Notoriously, the Tur\'an density of the complete $4$-vertex $3$-graph is unknown.

Instead of maximizing the number of edges, we may examine degree versions of the Tur\'an number. For $r=2$, this is equivalent to determining the function $\ex(n,F)$, but when $r \geq 3$ there are multiple interpretations of ``degree.'' In an $r$-graph, the {\it co-degree} of an $(r-1)$-set $S$ is the number of edges containing $S$. The minimum co-degree over all $(r-1)$-sets in $H$ is denoted $\delta_{r-1}(H)$.

Let ${\F}$ be a family of $r$-graphs.
 The {\it co-degree Tur\'an number} $\coex(n,\F)$ is the maximum value of $\delta_{r-1}(H)$ over all $n$-vertex $r$-graphs $H$ that do not contain a member of $\F$ as a subhypergraph. 
Mubayi and Zhao~\cite{Mubayicode} established the existence of the limit (called the {\it co-degree density of $\F$})
\[
\gamma(\F) : = \lim_{n \rightarrow \infty} \frac{\coex(n,\F)}{n}
\]
and proved several results concerning $\gamma(\F)$.

For a summary of these measures of extremality (and others) see the excellent survey of Balogh, Clemen, and Lidick\'y~\cite{l2norm}.

Motivated by the degree versions of the Erd\H os-Ko-Rado theorem and co-degree Tur\'an numbers, Balogh, Lemons and Palmer \cite{BLP} examined the notion of positive co-degree and they determined the maximum size of an intersecting $r$-graph with minimum positive co-degree at least $k$ (see also extensions by Spiro~\cite{Spiro}).

\begin{defn}
    The \emph{minimum positive co-degree} of a non-empty $r$-graph ${H}$, denoted $\delta_{r-1}^+( {H})$, is the maximum $k$ such that if $S$ is an $(r-1)$-set contained in a hyperedge of $ {H}$, then $S$ is contained in at least $k$ distinct hyperedges of $ {H}$. 
\end{defn}

In light of these results, we propose the problem to maximize the minimum positive co-degree of an $n$-vertex $r$-graph subject to avoiding some forbidden subhypergraph $F$. 

\begin{defn}
    Let ${\F}$ be a family of $r$-graphs. We define the {\it positive co-degree Tur\'an number} $\mathrm{co^+ex}(n, {\F})$ as the maximum positive co-degree $\delta_{r-1}^+(H)$ over all $n$-vertex $r$-graphs $H$ that do not contain a member of $\F$ as a subhypergraph.
\end{defn}
 
We will say that an $n$-vertex, ${F}$-free $r$-graph $H$ is an \textit{extremal $r$-graph for} $ {F}$ if $\delta_{r-1}^+(H) = \mathrm{co^+ex}(n, {F})$.

Define the {\it positive co-degree density} of a forbidden family $\F$ as
\[
\gamma^+(\F) := \lim_{n \rightarrow \infty} \frac{\mathrm{co^+ex}(n,\F)}{n}.
\]

It is not immediately clear that $\underset{n \rightarrow \infty}{\lim} {\mathrm{co^+ex}(n,\F)}/{n}$ should exist in general; however, the argument demonstrating the existence of $\gamma(F)$ can be adapted to demonstrate the existence of $\gamma^+(F)$. Indeed, this was done by Pikhurko \cite{P}. We give an alternate proof in Section~\ref{sec-secondary} using the hypergraph removal lemma and a supersaturation result. 

In Section~\ref{sec-secondary}, we shall also examine the possible values of $\gamma^+(F)$. In particular, we show that there is no family $\mathcal{F}$ of forbidden $3$-graphs such that $\gamma^+(\mathcal{F}) \in (0,1/3)$. This situation is in contrast to the behavior of $\gamma(\mathcal{F})$, which is shown not to jump
by Mubayi and Zhao in \cite{Mubayicode}. On the other hand, for $3$-graphs, there are no Tur\'an densities in the range $(0,2/9)$ (see~\cite{Er1, Keevashsurvey}).

Before deriving these general results, in Section~\ref{sec-primary} we determine (or bound) the positive co-degree densities of a number of concrete $3$-graphs. These results are summarized in the last two columns of the table below with comparisons to the corresponding bounds for Tur\'an density $\pi$ and co-degree density $\gamma$.
The individual $3$-graphs in the table are defined in the respective subsections of Section~\ref{sec-primary}.

Table entries for $\gamma^+$ without citations represent original results which are best-known; note that (following initial circulation of these results) several improvements have been reported by other authors for bounds on $\gamma^+$. We have also included these bounds in the table below, with citations. For the sake of completeness, even in cases where new bounds on $\gamma^+$ have been reported, we will include sections with our originally circulated (sometimes elementary) bounds.

\begin{center}
\begin{tabular}{|l||c|c|c|c|c|c|}
\hline
 $F$ & $\leq \pi(F) $  & $\pi(F) \leq $ & $\leq \gamma(F) $ & $\gamma(F) \leq   $ & $\leq \gamma^+(F)  $ & $\gamma^+(F) \leq $ \\
 \hline
 \hline
 \hyperref[k4-]{$K_4^-$} & 2/7 \cite{K43-extremalfrankl}  & 0.28689 \cite{flagmatic} & $1/4$ \cite{codegreeconj} & $1/4$  \cite{FalgasK4-}   & $ 1/3 $ & $  1/3 $\\
\hyperref[f5]{$F_5$} & 2/9  \cite{Bollobascancellative} & 2/9 \cite{F5Frankl} & 0 \cite{l2norm} & 0 \cite{l2norm} & $1/3$ & $1/3$\\
\hyperref[f32]{$F_{3,2}$} & 4/9  \cite{F32Mubayi} & 4/9 \cite{F32furedi} & 1/3 \cite{codF32falgas} & 1/3 \cite{codF32falgas} & $1/2$ & $1/2$\\
\hyperref[fano]{$\mathbb{F}$} & 3/4 \cite{VeraSos} & 3/4 \cite{FanoFuredi} & 1/2 \cite{MubayiFano} & 1/2 \cite{MubayiFano} & 2/3 & 2/3\\
 \hyperref[k4]{$K_4$} &5/9 \cite{MR177847} & 0.5615 \cite{BaberTuran} & 1/2 \cite{MR1829685} & 0.529 \cite{l2norm} & $1/2$ & 0.54296  \cite{volec} \\

\hyperref[f33]{$F_{3,3}$} & 3/4 \cite{F32Mubayi} & 3/4 \cite{F32Mubayi} & 1/2 \cite{l2norm} & 0.604 \cite{l2norm} & 3/5 & 0.616 \cite{BL} \\
\hyperref[c5c5-]{$C_{5}$} & $2\sqrt{3} - 3$ \cite{F32Mubayi} & 0.46829 \cite{flagmatic} & 1/3 \cite{l2norm} & 0.3993 \cite{l2norm} & $1/2$ & 1/2 \cite{Wu} \\
\hyperref[c5c5-]{$C_{5}^-$} &1/4 \cite{F32Mubayi} & 0.25074\cite{flagmatic} & 0 \cite{l2norm} & 0.136 \cite{l2norm} & 1/3 & 1/3 \cite{Wu} \\

 \hyperref[jk]{$J_4$} & $1/2$ \cite{DaisyBollobas} & $0.50409$ \cite{flagmatic} & $1/4$ \cite{l2norm} & $0.473$ \cite{l2norm} & $4/7$ & 0.58 \cite{BL} \\
\hline
\end{tabular}
\end{center}

Note that for all $3$-graphs $F$ which we consider, except for the $K_4$, we are able to show that $\gamma(F) \neq \gamma^+(F)$. The key difference is that in the positive co-degree setting, we are allowed co-degree zero pairs, and indeed our constructions feature large sets of vertices whose co-degrees are pairwise zero. We can often consider such constructions as analogous to blow-ups, which are not possible in the ordinary co-degree setting.

\section{Bounds for small 3-graphs}\label{sec-primary}

In this section, we will consider forbidden $3$-graphs which are not $3$-partite, i.e., do not appear in a $K_{n/3, n/3, n/3}$ for any $n$. We begin with some definitions which will be useful throughout the section. 

For $k \geq r$,  an $r$-graph $H$ is \textit{k-partite} if there exists a partition $V_1, V_2, \dots, V_k$ of its vertex set such that each $r$-edge intersects each partition class in at most one vertex. 
A $k$-partite $r$-graph $H$ is \textit{complete} if 
every possible edge is present and is {\it balanced} if the class sizes differ by at most $1$ (i.e.\ are as close in size as possible).

A $3$-graph $H$ is \textit{bipartite} if there is a partition $X,Y$ of its vertex set such that any $3$-edge of $H$ intersects both $X$ and $Y$. We say that $H$ is \textit{one-way bipartite} if there exists a partition $X,Y$ such that every $3$-edge is of the form $x_1x_2y$, where $x_1,x_2 \in X$ and $y \in Y$, and is \textit{complete one-way bipartite} if every $3$-edge of the form $x_1x_2y$ is present.

Let $H$ be an $r$-graph with vertex set $V$. We say that a subset $S$ of $V$ is \textit{independent} if every $r$-edge of $H$ includes a vertex in $V\setminus S$. We say that $S$ is \textit{strongly independent} if every $(r-1)$-subset of $S$ has co-degree zero, i.e., no $r$-edge intersects $S$ in at least $r-1$ vertices. Note that in a one-way bipartite graph one partition class is independent and the other is strongly independent.

Given an $3$-graph $H$ and vertices $u,v$ of $H$, the \textit{common neighborhood} of $u,v$, denoted by $N(u,v)$, is the set of vertices $w$ such that $uvw$ form a $3$-edge of $H$.

\subsection{\texorpdfstring{$K_4^-$}{TEXT}}\label{k4-}

\begin{center}
  \begin{longtable}{ | l|| c |  c | c  |   }
    \hline
    $H$ &   Edges   &  Figure 1 &  Figure 2   \\
    \hline\hline
    \hline
    $K_4^{-}$ 
    &123, 124, 134
    &
\vc{\begin{tikzpicture}\outercycle{5}{4}
\draw (x0) node[unlabeled_vertex]{};\draw (x1) node[unlabeled_vertex]{};\draw (x2) node[unlabeled_vertex]{};\draw (x3) node[unlabeled_vertex]{};
\drawhyperedge{0}{4}
\drawhypervertex{0}{0}
\drawhypervertex{1}{0}
\drawhypervertex{2}{0}
\drawhyperedge{1}{4}
\drawhypervertex{0}{1}
\drawhypervertex{1}{1}
\drawhypervertex{3}{1}
\drawhyperedge{2}{4}
\drawhypervertex{0}{2}
\drawhypervertex{2}{2}
\drawhypervertex{3}{2}

\path (0,0.4) -- (0,-1.1); 
\end{tikzpicture} 
}    
    &
   \Kfourthreeminus
   \\
\hline
  \end{longtable}
\end{center}

We will begin by determining $\mathrm{co^+ex}(n,K_4^-)$. The first theorem which we state, due to Frankl and F\"uredi \cite{K43-extremalfrankl}, characterizes a special family of $3$-graphs which will provide constructions achieving $\mathrm{co^+ex}(n,K_4^-)$. Before stating the theorem, we will need to introduce a pair of families of constructions from \cite{K43-extremalfrankl}. 

\begin{const} Let $H$ consist of $n$ vertices placed on the unit circle, with $3$-edges those triples $xyz$ such that the triangle with vertices $x,y,z$ contains the origin. (We may assume that no pair of vertices are placed so that the line connecting them passes through the origin.)
\end{const}

\begin{const} Let $ {H}_6$ be the unique $(6,3,2)$-design, i.e., the $3$-graph on vertex set $\{1,2,3,4,5,6\}$ with $3$-edge set
$$ E_6 = \{ 123, 124, 345, 346, 561, 562, 135, 146, 236, 245 \}.$$
We take an $n$-vertex blow-up of $ {H}_6$, that is, a $3$-graph $ {H}$ with vertex set $V$ of size $n$ partitioned as $V = V_1 \cup \dots \cup V_6$, and $3$-edge set 
$$ {E} = \{ v_{i_1}v_{i_2}v_{i_3}: i_1 < i_2 < i_3; \text{ for all } i_j, v_{i_j} \in V_{i_j} ; i_1 i_2 i_3 \in  E_6\}.$$
\end{const}

\begin{const} Let $ {H}_6$ be the unique $(6,3,2)$-design, i.e., the $3$-graph on vertex set $\{1,2,3,4,5,6\}$ with $3$-edge set
$$ E_6 = \{ 123, 124, 345, 346, 561, 562, 135, 146, 236, 245 \}.$$
We take an $n$-vertex blow-up of $ {H}_6$, that is, a $3$-graph $ {H}$ with vertex set $V$ of size $n$ partitioned as $V = V_1 \cup \dots \cup V_6$, and $3$-edge set 
$$ {E} = \{ v_{i}v_{j}v_{k}: \text{ for all } v_i \in V_i, v_{j} \in V_{j}, v_k \in V_k \text{ such that } ijk \in  E_6\}.$$
\end{const}

We depict $ H_6$ in Figure~\ref{fig1}.

\begin{figure}[ht]
\begin{center}

\resizebox{0.3\textwidth}{!}{
\vc{\begin{tikzpicture}\outercycle{6}{5} \outercycle{5}{4}
\filldraw (0,0) circle(0.05 cm);
\filldraw (0.5,0) circle(0.05 cm);
\filldraw (1,0) circle(0.05 cm);
\filldraw (1.5,0) circle(0.05 cm);
\filldraw (2,0) circle(0.05 cm);
\filldraw (2.5,0) circle(0.05 cm);
\drawhyperedge{0}{6}
\drawhypervertex{0}{0}
\drawhypervertex{1}{0}
\drawhypervertex{2}{0}

\drawhyperedge{1}{6}
\drawhypervertex{0}{1}
\drawhypervertex{1}{1}
\drawhypervertex{3}{1}

\drawhyperedge{2}{6}
\drawhypervertex{2}{2}
\drawhypervertex{3}{2}
\drawhypervertex{4}{2}

\drawhyperedge{3}{6}
\drawhypervertex{2}{3}
\drawhypervertex{3}{3}
\drawhypervertex{5}{3}

\drawhyperedge{4}{6}
\drawhypervertex{0}{4}
\drawhypervertex{4}{4}
\drawhypervertex{5}{4}

\drawhyperedge{5}{6}
\drawhypervertex{1}{5}
\drawhypervertex{4}{5}
\drawhypervertex{5}{5}

\drawhyperedge{6}{6}
\drawhypervertex{0}{6}
\drawhypervertex{2}{6}
\drawhypervertex{4}{6}

\drawhyperedge{7}{6}
\drawhypervertex{0}{7}
\drawhypervertex{3}{7}
\drawhypervertex{5}{7}

\drawhyperedge{8}{6}
\drawhypervertex{1}{8}
\drawhypervertex{2}{8}
\drawhypervertex{5}{8}

\drawhyperedge{9}{6}
\drawhypervertex{1}{9}
\drawhypervertex{3}{9}
\drawhypervertex{4}{9}

\path (0,0.4) -- (0,-1.1); 
\end{tikzpicture} 
}   
}

\caption{The $3$-edges of $H_6$.}\label{fig1}
\end{center}
\end{figure}
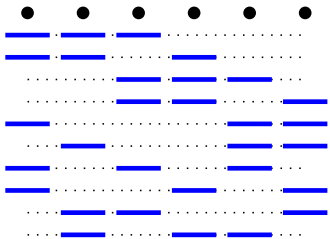

It is easy to see that the $3$-graphs in Constructions 1 and 2 have the property that any four vertices span $0$ or $2$ $3$-edges. In fact, the following theorem says that these are the only $3$-graphs with this property.

\begin{thm}[Frankl-F\"uredi, \cite{K43-extremalfrankl}] \label{FF}
    Suppose $ {H}$ is a $3$-graph in which any $4$ vertices span $0$ or $2$ $3$-edges. Then $ {H}$ is isomorphic to one of the $3$-graphs in Construction 1 or 2.
\end{thm}

We are now ready to state our theorem for $K_4^-$.

\begin{thm}
\[
\mathrm{co^+ex}(n,K_4^-) = \Big\lfloor \frac{n}{3} \Big\rfloor.
\]
Moreover, when $n \equiv 0 \pmod 6$, then there are exactly two extremal constructions, namely the complete balanced $3$-partite $3$-graph and the blow-up of $ {H}_6$ whose class sizes are balanced; when $n \equiv 3 \pmod 6$, the unique extremal construction is the complete balanced $3$-partite $3$-graph. 

\end{thm}

\begin{proof}

We first show that $\mathrm{co^+ex}(n,K_4^-) \leq \lfloor n/3 \rfloor$. Let $H$ be a $K_4^-$-free 3-graph on $n$ vertices. Fix a $3$-edge $xyz$ of $H$. It is easy to see that $N(x,y), N(y,z),$ and $N(x,z)$ must be pairwise disjoint; indeed, suppose there is a vertex $w \in N(x,y) \cap N(y,z)$. Then $xyz, xyw, yzw$ are three hyperedges spanning four vertices, i.e., a copy of $K_4^-$, a contradiction. Thus, we must have 
$$|N(x,y)| + |N(y,z)| + |N(x,z)| \leq n.$$
This immediately implies that one of $N(x,y), N(y,z), N(x,z)$ has size at most $\lfloor n/3 \rfloor$.  

Note that it is possible to achieve $\delta_2^+(H) = \lfloor n/3 \rfloor$ for any $n$, since the complete balanced $3$-partite $3$-graph on $n$ vertices has minimum positive co-degree $\lfloor n/3 \rfloor$ and contains no copy of $K_4^-$. Thus, all that is left is to determine whether this extremal construction is unique. We do this when $n$ is divisible by 3. 

Suppose $n$ is divisible by $3$ and $H$ a is $K_4^-$-free $3$-graph on $n$ vertices with $\delta_2^+(H) = n/3$. We claim that any $4$ vertices of $H$ span either 0 or 2 $3$-edges. Indeed, take $4$ vertices $x,y,z,w$ of $H$. If these span 0 $3$-edges, we are done, so suppose not. Thus, $x,y,z,w$ span at least one $3$-edge; without loss of generality, $xyz$ is a $3$-edge. We have seen that $N(x,y),N(y,z),$ and $N(x,z)$ are pairwise disjoint, and all must have size at least $n/3$ to satisfy the minimum positive co-degree condition on $H$. This implies that all have size exactly $n/3$, and the vertex set is partitioned by $N(x,y),N(y,z),N(x,z)$. In particular, $w$ is in (exactly) one of $N(x,y),N(y,z), N(x,z)$, showing that $x,y,z,w$ span exactly 2 $3$-edges.

We can now apply Theorem \ref{FF} to conclude that $H$ is isomorphic to one of the $3$-graphs described in Constructions $1$ and $2$. Among these $3$-graphs, we claim that at most two have minimum positive co-degree $n/3$. 

We first consider $3$-graphs of the type described in Construction $1$. Let $H$ be a $3$-graph in the family described by Construction $1$, and suppose that $\delta_2^+(H) = n/3$. We claim that $H$ is the complete balanced $3$-partite $3$-graph.

To prove this claim, let $d(X,Y)$ denote the distance on the circle between vertices $X,Y$. Choose vertices $X,Y$ with positive co-degree such that
$$d(X,Y) = \min \{d(A,B): A,B \text{ have positive co-degree}\}.$$
Denote by $X', Y'$ the antipodes of $X,Y$. Recall that we assume that no two vertices lie on a line through the origin, so the points $X'$ and $ Y'$ are not vertices of $H$. 

We will write $UW$ to denote the (minor) arc between two points $U,W$ on the circle. For a vertex $V$ of $H$, we write $V \in UW$ if $V$ lies in the arc $UW$.

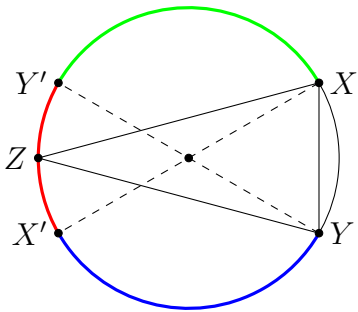
\begin{figure}[ht]
\begin{center}

\begin{tikzpicture}

\draw (0,0) circle (2 cm);

\draw[green, very thick] ({2*cos(30)}, {2*sin(30)}) arc(30:150:2);

\draw[red, very thick] ({2*cos(150)}, {2*sin(150)}) arc(150:210:2);

\draw[blue, very thick] ({2*cos(210)}, {2*sin(210)}) arc(210:330:2);

\filldraw (0,0) circle(0.05 cm);

\filldraw ({2*cos(30)}, {2*sin(30)}) circle(0.05 cm);
\filldraw ({2*cos(-30)}, {2*sin(-30)}) circle(0.05 cm);
\draw ({2*cos(30)}, {2*sin(30)}) node[right]{$X$};
\draw ({2*cos(-30)}, {2*sin(-30)}) node[right]{$Y$};

\filldraw ({2*cos(210)}, {2*sin(210)}) circle(0.05 cm);

\filldraw ({2*cos(-210)}, {2*sin(-210)}) circle(0.05 cm);

\draw ({2*cos(-210)}, {2*sin(-210)}) node[left]{$Y'$};
\draw ({2*cos(210)}, {2*sin(210)}) node[left]{$X'$};

\filldraw (-2,0) circle(0.05 cm);
\draw (-2,0) node[left]{$Z$};

\draw[dashed] ({2*cos(30)}, {2*sin(30)}) -- ({2*cos(210)}, {2*sin(210)});
\draw[dashed] ({2*cos(-30)}, {2*sin(-30)}) -- ({2*cos(-210)}, {2*sin(-210)});

\draw (-2,0) -- ({2*cos(30)}, {2*sin(30)}) -- ({2*cos(-30)}, {2*sin(-30)}) -- (-2,0);

\end{tikzpicture}

\end{center}

\caption{The unit circle with three relevant arcs}\label{fig2}

\end{figure}

Observe that $N(X,Y)$ is precisely the set of vertices which lie in $X'Y'$. Thus, at least $n/3$ vertices lie in $X'Y'$. 

Consider a vertex $Z \in X'Y'$. Since $X', Y'$ are not vertices of $H$, any two vertices which lie in $X'Y'$ have distance strictly smaller than $d(X,Y)$, and therefore have co-degree zero. Thus, $N(X,Z)$ is disjoint from $X'Y'$, so must be contained in $X'Y$. We conclude that at least $n/3$ vertices of $H$ lie in $X'Y$. Analogously, at least $n/3$ vertices of $H$ lie in $XY'$. Thus, exactly $n/3$ vertices lie in each of the three arcs $XY', X'Y',$ and $X'Y$. 

Note that the above argument also shows that $Z$ has positive co-degree with every vertex in $XY'$ and every vertex in $X'Y$. Now, take $U \in XY'$ and consider $N(U,Z)$. Observe that $N(UZ) \subset X'Y$, so in fact to achieve co-degree $n/3$, we must have $V \in N(U,Z)$ for every $V \in X'Y$. Since $Z,U$ are chosen arbitrarily, this implies that $H$ contains $K_{n/3,n/3,n/3}$ as a sub-hypergraph. Thus, $H$ is $K_{n/3,n/3,n/3}$, since any 3-edge added to $K_{n/3,n/3,n/3}$ yields a $K_4^-$.

Now consider Construction 2. By definition, a 3-graph of this type has six vertex classes $V_1, \dots, V_6$, and any pair of vertices $x,y$ with positive co-degree have $N(x,y) = V_i \cup V_j$ for some $i \neq j$. Thus, if every class has size exactly $n/6$, then the minimum positive co-degree is exactly $2n/6 = n/3$. This can occur when $n \equiv 0 \pmod 6$, if we choose to balance the classes. However, when $n \equiv 3 \pmod 6$, even if we make the construction as balanced as possible, three classes will be of size $\lfloor n/6 \rfloor$. Construction 2 also has the property that for any pair $i,j$ from $\{1, \dots , 6\}$ (with $i \neq j$), there exists a pair of vertices $x,y$ with $N(x,y) = V_i \cup V_j$. Thus, when $n \equiv 3 \pmod 6$, we can find a pair of vertices $x,y$ with $|N(x,y)| \leq 2 \lfloor n/6 \rfloor < n/3$. Hence, Construction 2 yields an extremal $3$-graph when $n \equiv 0 \pmod 6$, but not when $n \equiv 3 \pmod 6$. 
\end{proof}

We do not attempt a uniqueness result for $n \not\equiv 0 \pmod 3$; in this case, Theorem \ref{FF} may not apply, since given minimum positive co-degree $\lfloor n/3\rfloor < n/3$, it is possible that some $4$-sets of vertices span exactly one edge. When $n \not\equiv 0 \pmod 3$, the complete balanced $3$-partite $3$-graph is extremal, but small perturbations of the construction are also acceptable. For instance, a single vertex can be removed from a class of size $\lceil n/3 \rceil$ and isolated without altering the minimum positive co-degree. Construction $2$ may or may not yield an extremal $3$-graph, depending upon the congruence class of $n \pmod 6$. Again, when Construction 2 yields an extremal $3$-graph, small perturbations may be possible. While it is plausible that only $3$-graphs of the types described in Constructions 1 and 2, and small perturbations thereof, yield extremal $3$-graphs for $K_4^-$ in general, we also do not rule out the possibility that substantially different constructions are extremal when $n$ is not divisible by $3$.

Due to the structure of the $K_{4}^{-}$, a supersaturation result also follows quickly. We will first state a lemma of general applicability.

\begin{lemma}\label{edge approx}
Fix $c>0$ and
suppose $ {H}$ is an $r$-graph with $\delta_{r-1}^+(H) \geq cn$. Then, for $n$ large enough, $|E( {H})| \geq \frac{1}{2}\frac{c^r}{r!} n^r$.
\end{lemma}

\begin{proof}
Let $f(H)$ be the number of $(r-1)$-sets in $H$ with positive co-degree. Observe that 
$$f(H)cn \leq f(H) \delta_{r-1}^+(H) \leq r|E(H)|,$$
so it suffices to estimate $f(H)$. Since $\delta_{r-1}^+(H) \geq cn > 0$, there is an $(r-1)$-set $T$ in $V(H)$ with positive co-degree. Consider the following process for building $(r-1)$-sets with positive co-degree, starting from $T$. Set $T_1 = T$. At step $i$, choose a vertex $t_i$ in $T_i \cap T$ and a vertex $v_i \in V(H) \setminus T$ such that $T_i \cup \{v_i\}$ span an $r$-edge of $H$. Set $T_{i+1} = (T_i \setminus\{t_i\}) \cup \{v_i\}$. If $T_{i} \cap T = \emptyset$, we end the process and return $T_i$. Observe that, by construction, all $T_i$'s have positive co-degree. Thus, at each step, we will be able to find a vertex $v_i$ to add to $T_i$. Observe also that the process will terminate at a set $T_{r}$ such that $T_{r} \cap T = \emptyset$. Now, we count the number of ways to build $T_{r}$.

At step $i$, there are at least $cn - i + 1$ choices for $v_i$, since $d_{r-1}(T_i) \geq cn$ and $T$ contains $i-1$ vertices which are not in $T_i$. A set $T_{r}$ may be generated up to $(r-1)!$ times by this process. Therefore, the number of distinct $(r-1)$-sets $T_r$ of positive co-degree is at least 
$$\frac{1}{(r-1)!}\prod_{i=1}^{r-1} (cn - i + 1) \geq \frac{c^{r-1}}{2(r-1)!}n^{r-1}$$
for large enough $n$. Thus, $|E(H)| \geq \frac{1}{2}\frac{c^r}{r!}n^r$ for large enough $n$.
\end{proof}

We can now prove a supersaturation result for $K_4^-$.
We will later prove a more general positive co-degree supersaturation result. However, the following theorem has several nice properties: it gives an explicit value of $\delta$ which is relatively large in terms of $\varepsilon$, and it directly leverages the structure of $K_4^-$ for an elementary proof.

\begin{thm}
Fix $\varepsilon > 0$. If $ {H}$ has $\delta_2^+( {H}) = (1/3 + \varepsilon)n$, then there exists $\delta > 0$ such that $ {H}$ contains at least $\delta n^4$ copies of $K_4^-$.  

\end{thm}

\begin{proof}
We claim that $\delta = \frac{\varepsilon}{324}$ is sufficient. 

We first show that each $3$-edge of $ {H}$ is contained in at least $3 \varepsilon n$ copies of $K_4^-$. Indeed, suppose $abc$ is a $3$-edge, and consider $N(ab), N(bc), N(ac)$. The number of copies of $K_4^-$ which use the $3$-edge $abc$ is precisely 
$|N(ab) \cap N(bc)| + |N(ab) \cap N(ac)| + |N(bc) \cap N(ac)|,$
where the first term in the sum counts the number of copies of $K_4^-$ of the form $abc, abd, bcd$, and so on. Observe,
$$n \geq |N(ab) \cup N(bc) \cup N(ac)|$$
$$\geq |N(ab)| + |N(bc)| + |N(ac)| - \Big( |N(ab) \cap N(bc)| + |N(ab) \cap N(ac)| + |N(bc) \cap N(ac)|\Big) $$
$$\geq 3\left(\frac{n}{3} + \varepsilon n\right) - \Big( |N(ab) \cap N(bc)| + |N(ab) \cap N(ac)| + |N(bc) \cap N(ac)|\Big)$$
$$= n + 3\varepsilon n - \Big( |N(ab) \cap N(bc)| + |N(ab) \cap N(ac)| + |N(bc) \cap N(ac)|\Big),$$
which implies that 
$$ |N(ab) \cap N(bc)| + |N(ab) \cap N(ac)| + |N(bc) \cap N(ac)| \geq 3 \varepsilon n.$$

We now count copies of $K_4^-$ in $ {H}$ by $3$-edge. Since each $3$-edge in $ {H}$ is contained in at least $3 \varepsilon n$ copies of $K_4^-$, and each copy of $K_4^-$ will be counted exactly $3$ times, we know we have at least $\varepsilon n \cdot |E(H)|$ copies of $K_4^-$. By Lemma~\ref{edge approx}, with $c = 1/3 + \varepsilon$, we have 
$$|E(H)| \geq \frac{(1/3 + \varepsilon)^3}{12}n^3 > \frac{1}{324}n^3.$$
So $ {H}$ contains at least $\frac{\varepsilon}{324}n^4$ copies of $K_4^-$.
\end{proof}

\subsection{\texorpdfstring{$F_5$}{text}}\label{f5}

\begin{center}
  \begin{longtable}{ | l ||  c  |  c  |  c  |   }
    \hline
$F_5$
&
123, 124, 345
&
\vc{\begin{tikzpicture}\outercycle{6}{5}
\draw (x0) node[unlabeled_vertex]{};\draw (x1) node[unlabeled_vertex]{};\draw (x2) node[unlabeled_vertex]{};\draw (x3) node[unlabeled_vertex]{};\draw (x4) node[unlabeled_vertex]{};
\drawhyperedge{0}{5}
\drawhypervertex{0}{0}
\drawhypervertex{1}{0}
\drawhypervertex{2}{0}
\drawhyperedge{1}{5}
\drawhypervertex{0}{1}
\drawhypervertex{1}{1}
\drawhypervertex{3}{1}
\drawhyperedge{2}{5}
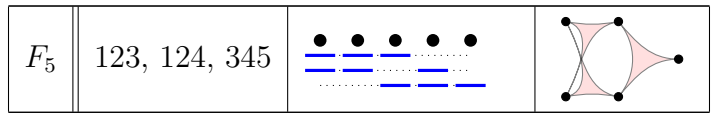
\drawhypervertex{2}{2}
\drawhypervertex{3}{2}
\drawhypervertex{4}{2}
\end{tikzpicture} 
}
&
\vc{
\begin{tikzpicture}
\path (0,0.7) -- (0,-0.7); 
\draw
(0, 0.5) coordinate(0) node[vtx]{}
(0, -0.5) coordinate(1) node[vtx]{}
(0.7, 0.5) coordinate(2) node[vtx]{}
(0.7,-0.5) coordinate(3) node[vtx]{}
(1.5, 0) coordinate(4) node[vtx]{}
;
\draw[hyperedge] (0) to[out=-10,in=190,looseness=1] (2) to[out=190,in=60,looseness=1] (1) to[out=60,in=-10,looseness=1] (0);
\draw[hyperedge] (1) to[out=10,in=170,looseness=1] (3) to[out=170,in=-60,looseness=1] (0) to[out=-60,in=10,looseness=1] (1);
\draw[hyperedge] (2) to[out=310,in=50,looseness=1] (3) to[out=50,in=180,looseness=1] (4) to[out=180,in=310,looseness=1] (2);
\end{tikzpicture}
} 
   \\
\hline
  \end{longtable}
\end{center}

The following theorem gives an exact result for graphs on at least $6$ vertices. Note that the statement does not hold for values smaller than $6$; for $n = 4$, the $K_4$ has minimum positive co-degree $2 > \floor{4/3}$ and no $F_5$, while for $n = 5$, the graph obtained by adding a single isolated vertex to a $K_4$ has minimum positive co-degree $2 > \floor{5/3}$ and no $F_5$.

\begin{thm}
For $n \geq 6$, 
$$\mathrm{co^+ex}(n,F_5) = \floor{\frac{n}{3}}.$$
\end{thm}

\begin{proof}

The complete balanced $3$-partite $3$-graph is easily checked to be $F_5$-free, and has minimum positive co-degree $\floor{n/3}$. So, we need only supply a matching upper bound.

Suppose $ {H}$ has at least $6$ vertices and $\delta_2^+( {H}) > \floor{n/3}$. We claim that $ {H}$ must contain an $F_5$. Observe that, since $n \geq 6$, we have $\floor{n/3} \geq 2$, so $\delta_2^+( {H}) > 2$. Consider a $3$-edge $abc$ of $ {H}$. Since $\delta^+( {H}) > \floor{n/3}$, it follows that at least two of $N(ab),N(bc),N(ac)$ have non-empty intersection. Without loss of generality, there exists $d$ in $N(ab) \cap N(bc)$. So the pair $c,d$ has posititve co-degree. Since $\delta^+( {H}) \geq 3$, this means that there is a new vertex $e$ such that $cde$ is a $3$-edge. But now $abc, abd, cde$ form an $F_5$-copy. 
\end{proof}

When $n$ is divisible by $3$, we can also easily show that the complete balanced $3$-partite $3$-graph is the unique extremal 3-graph for $F_5$. Indeed, suppose $H$ is a $3$-graph on $n \geq 6$ vertices, where $3$ divides $n$, with minimum positive co-degree $n/3$. Let $abc$ be a $3$-edge of $H$. As noted above, $N(a,b), N(a,c),$ and $N(b,c)$ are pairwise disjoint, so form a balanced triparition of the vertex set. Thus, $H$ is a $3$-partite $3$-graph with parts $N(a,b)$, $N(a,c)$, and $N(b,c)$. It is clear that in order to achieve the co-degree condition, $H$ must be complete.

\subsection{\texorpdfstring{$F_{3,2}$}{text}}\label{f32}

\begin{center}
  \begin{longtable}{ | l ||  c  |  c  |  c  |   }
    \hline
$F_{3,2}$ & 123, 145, 245, 345
&
\vc{\begin{tikzpicture}\outercycle{6}{5}
\draw (x0) node[unlabeled_vertex]{};\draw (x1) node[unlabeled_vertex]{};\draw (x2) node[unlabeled_vertex]{};\draw (x3) node[unlabeled_vertex]{};\draw (x4) node[unlabeled_vertex]{};
\drawhyperedge{0}{5}
\drawhypervertex{0}{0}
\drawhypervertex{1}{0}
\drawhypervertex{2}{0}
\drawhyperedge{1}{5}
\drawhypervertex{0}{1}
\drawhypervertex{3}{1}
\drawhypervertex{4}{1}
\drawhyperedge{2}{5}
\drawhypervertex{1}{2}
\drawhypervertex{3}{2}
\drawhypervertex{4}{2}
\drawhyperedge{3}{5}
\drawhypervertex{2}{3}
\drawhypervertex{3}{3}
\drawhypervertex{4}{3}
\end{tikzpicture} 
}
&
\vc
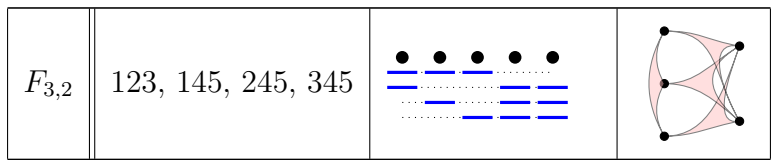
{
\begin{tikzpicture}
\path (0,1) -- (0,-1); 
\draw
(0,-0.7) coordinate(1) node[vtx](a){}
(0,0) coordinate(2) node[vtx](b){}
(0,0.7) coordinate(3) node[vtx](c){}
(1,-0.5) coordinate(4) node[vtx](d){}
(1,0.5) coordinate(5) node[vtx](d){}
;
\draw[hyperedge] (1) to[bend left] (2) to[bend left] (3) to[bend right] (1);
\draw[hyperedge] (1) to[out=0,in=190] (4) to[out=190,in=250] (5) to[out=250,in=0] (1);
\draw[hyperedge] (2) to[out=0,in=140] (4) to[out=140,in=210,looseness=1.5] (5) to[out=210,in=0] (2);
\draw[hyperedge] (3) to[out=0,in=120,looseness=1.2] (4) to[out=120,in=170,looseness=1.2] (5) to[out=170,in=0,looseness=1.2] (3);
\end{tikzpicture}
}
   \\
\hline
  \end{longtable}
\end{center}

\begin{thm}\label{F32}
\[
\mathrm{co^+ex}(n,F_{3,2}) \leq   \frac{n}{2}. 
\]

 Moreover, if $H$ is an $F_{3,2}$-free $n$-vertex $3$-graph with $\delta_2^+(H) = n/2$, then $n$ is divisible by 4 and $H$ is the complete balanced $4$-partite $3$-graph. 
\end{thm}

\begin{proof}
Let $H$ be an $n$-vertex $3$-graph with minimum positive co-degree $\delta^+_2(H) \geq n/2$.
Let $a,b$ be a pair of vertices with positive co-degree. Then $|N(a,b)| \geq n/2$. 
Let $c \in N(a,b)$ and consider $N(a,c)$. The common neighborhood $N(a,c)$ does not include $a$ and has size at least $n/2$, so $N(a,b) \cap N(a,c) \not = \emptyset$. So let $d$ be a vertex in $N(a,b) \cap N(a,c)$.
There is no $3$-edge in $N(a,b)$ as $H$ is $F_{3,2}$-free. Therefore, $N(c,d) \cap N(a,b) = \emptyset$. 
Thus, as $|N(a,b)| + |N(c,d)| \leq n$ and $|N(c,d)| \geq n/2$ , we have
$|N(a,b)| \leq n/2$.

Now, if $\delta_2^+(H) = n/2$, then we can find $a,b,c,d$ as above, so that $N(a,b)$ and $N(c,d)$ partition the vertex set of $H$. Note that $N(a,c)$ and $N(b,d)$ also partition the vertex set of $H$, so we have a partition into four classes, $N(a,b) \cap N(a,c)$, $N(a,b) \cap N(b,d)$, $N(c,d) \cap N(a,c)$, and $N(c,d) \cap N(b,d)$. Since the common neighborhood of any pair of vertices with positive co-degree must be independent to avoid an $F_{3,2}$, each of these four parts is independent. We claim, moreover, that each part is strongly independent. 

It will suffice to show that 
$$N(b,c) = [N(c,d) \cap N(b,d)] \cup [N(a,b) \cap N(a,c)]$$
and
$$N(a,d) = [N(c,d) \cap N(a,c)] \cup [N(a,b) \cap N(b,d)],$$
for then we will have established that the union of any two parts is the common neighborhood of a pair of vertices, so is independent, and therefore every $3$-edge in $H$ must use vertices from three parts.

Suppose, for a contradiction, that $x \in N(b,c)$ is contained in $N(a,b) \cap N(b,d)$. Then $c$ and $x$ have positive co-degree. Both $c$ and $x$ are in $N(a,b)$, so $N(c,x) \cap N(a,b) = \emptyset$, and thus $N(c,x) = N(c,d)$. In particular, $acx$ is a $3$-edge, so $x$ is in $N(a,c)$, as are $b$ and $d$. But by assumption, $x \in N(b,d)$, a contradiction, as $N(a,c)$ must be independent to avoid an $F_{3,2}$. Thus, $N(b,c)$ does not intersect $N(a,b) \cap N(b,d)$.

Similarly, $N(b,c)$ does not intersect $N(c,d) \cap N(a,c)$. We conclude that $$N(b,c) \subseteq [N(c,d) \cap N(b,d)] \cup [N(a,b) \cap N(a,c)].$$ An analogous argument shows that $$N(a,d) \subseteq [N(c,d) \cap N(a,c)] \cup [N(a,b) \cap N(b,d)].$$ Together, these containments imply that 
$$N(b,c) = [N(c,d) \cap N(b,d)] \cup [N(a,b) \cap N(a,c)]$$
and
$$N(a,d) = [N(c,d) \cap N(a,c)] \cup [N(a,b) \cap N(b,d)],$$
since each of $[N(c,d) \cap N(b,d)] \cup [N(a,b) \cap N(a,c)]$, $[N(c,d) \cap N(b,d)] \cup [N(a,b) \cap N(a,c)]$ must have size exactly $n/2$ to allow both $|N(b,c)| \geq n/2$ and $|N(a,d)| \geq n/2$. 

We conclude that $H$ is $4$-partite, so it only remains to show that $H$ is complete and balanced. Firstly, from the fact that the union of any two parts is the common neighborhood of a pair of vertices, it follows that the union of any two parts has size exactly $n/2$. From this, it is simple to deduce that all class sizes are equal, so $H$ is balanced (and $n$ must be divisible by 4). Now, to show that $H$ is complete, it suffices to show that any two vertices $x,y$ which lie in different parts have positive co-degree. Consider a vertex $x$ of $H$. Without loss of generality, $x$ is in $N(c,d) \cap N(b,d)$. So $x$ has positive co-degree with $b$ and $c$. Since $N(x,b)$ does not intersect $N(c,d)$, we must have $N(x,b) = N(a,b)$, so $x$ has positive co-degree with every vertex in $N(a,b)$. Similarly, $N(x,c)$ must not intersect $N(b,d)$, so $N(x,c) = N(a,c)$, and $x$ has positive co-degree with every vertex in $N(a,c)$. This shows that $x$ has positive co-degree with every vertex which is not in $N(c,d) \cap N(b,d)$. We conclude that $H$ is complete.
\end{proof}

We can now quickly characterize the exact value of $\mathrm{co^+ex}(n,F_{3,2})$ for all $n$.

\begin{cor}
If $4\not| n$, then
\[
 \mathrm{co^+ex}(n,F_{3,2}) = \left \lfloor \frac{n-1}{2} \right \rfloor.
\]

\end{cor}

\begin{proof}

By Theorem \ref{F32}, we know that if $4 \not | n$, then $ \mathrm{co^+ex}(n,F_{3,2}) < \frac{n}{2}$. Note that when $n$ is odd, $ \left \lfloor \frac{n-1}{2} \right \rfloor =  \left \lfloor \frac{n}{2} \right \rfloor$, and when $n \equiv 2 \pmod 4$,  $ \left \lfloor \frac{n-1}{2} \right \rfloor = \frac{n}{2} - 1$. So in both cases, $ \left \lfloor \frac{n-1}{2} \right \rfloor$ is the largest integer which is strictly less than $\frac{n}{2}$. Thus, Theorem \ref{F32} tells us that when $4 \not | n$, we have  $\mathrm{co^+ex}(n,F_{3,2}) \leq \left \lfloor \frac{n-1}{2} \right \rfloor$. 

For a matching lower bound, observe that the complete $4$-partite $3$-graph whose class sizes are as balanced as possible has minimum positive co-degree $\left \lfloor \frac{n-1}{2} \right \rfloor$ when $4 \not | n$.
\end{proof}

While we have fully characterized the values of $\mathrm{co^+ex}(n,F_{3,2})$, we may still ask about uniqueness of extremal constructions when $4$ does not divide $n$. In these cases, multiple constructions achieve a minimum positive co-degree of $\lfloor \frac{n-1}{2} \rfloor $; moreover, the situation seems to vary based on the congruence class of $n$ modulo $4$. We do not fully characterize the families of extremal 3-graphs in these cases, but mention some alternate extremal constructions. For all $n$ which are not congruent to $0$ modulo $4$, the complete one-way bipartite graph with balanced classes is also extremal (it is easily checked to be $F_{3,2}$-free). Depending on the congruence class of $n$ modulo $4$, small perturbations of the complete $4$-partite $3$-graph and the complete one-way bipartite $3$-graph may also be extremal. For example, a vertex can in some cases be removed from one of the classes and left as an isolated vertex without lowering the minimum positive co-degree. While we do not claim to describe all constructions, we believe that it would be possible to do so.


\subsection{Fano plane \texorpdfstring{$\mathbb{F}$}{text}}\label{fano}

\begin{center}
  \begin{longtable}{ | l ||  c  |  c  |  c  |   }
    \hline
    $\mathbb{F}$ 
 &
 123, 345, 156, 246, 147, 257, 367
 &
\vc{\begin{tikzpicture}\outercycle{8}{7}
\draw (x0) node[unlabeled_vertex]{};\draw (x1) node[unlabeled_vertex]{};\draw (x2) node[unlabeled_vertex]{};\draw (x3) node[unlabeled_vertex]{};\draw (x4) node[unlabeled_vertex]{};\draw (x5) node[unlabeled_vertex]{};\draw (x6) node[unlabeled_vertex]{};
\drawhypervertex{0}{0}
\drawhypervertex{1}{0}
\drawhypervertex{2}{0}
\drawhyperedge{1}{7}
\drawhypervertex{0}{1}
\drawhypervertex{3}{1}
\drawhypervertex{6}{1}
\drawhyperedge{2}{7}
\drawhypervertex{0}{2}
\drawhypervertex{4}{2}
\drawhypervertex{5}{2}
\drawhyperedge{3}{7}
\drawhypervertex{1}{3}
\drawhypervertex{3}{3}
\drawhypervertex{5}{3}
\drawhyperedge{4}{7}
\drawhypervertex{1}{4}
\drawhypervertex{4}{4}
\drawhypervertex{6}{4}
\drawhyperedge{5}{7}
\drawhypervertex{2}{5}
\drawhypervertex{3}{5}
\drawhypervertex{4}{5}
\drawhyperedge{6}{7}
\drawhypervertex{2}{6}
\drawhypervertex{5}{6}
\drawhypervertex{6}{6}
\end{tikzpicture} 
}
&
\vc{
\begin{tikzpicture}
\path (60:2.2) -- (0,-0.3); 
\draw
(60:2) coordinate(1) coordinate(7) node[vtx]{}
(60:1) coordinate(2) coordinate(8) node[vtx]{}
(0,0) coordinate(3) node[vtx]{}
(1,0) coordinate(4) node[vtx]{}
(2,0) coordinate(5) node[vtx]{}
++ (120:1) coordinate(6) node[vtx]{}
(30:1.2) coordinate(X) node[vtx]{}
;
\foreach \i in {1,3,5}{
\pgfmathtruncatemacro{\j}{\i+1}
\pgfmathtruncatemacro{\k}{\i+2}
\pgfmathtruncatemacro{\l}{\i+3}
\draw[hyperedge] (\i) to[bend right=10] (\j) to[bend right=10] (\k) to[bend left=10] (\i) ;

\draw[hyperedge] (\i) to[bend right=10] (X) to[bend right=10] (\l) to[bend left=10] (\i) ;

}

\draw[hyperedge] (2) to[bend right=40] (4) to[bend right=40] (6) to[bend left=80,looseness=2.5] (2) ;

\draw
\foreach \i in {1,2,3,4,5,6,X}{
(\i) node[vtx]{}
}
;

\end{tikzpicture}
} 
   \\
   \hline
  \end{longtable}
\end{center}

\begin{thm}
\[
\mathrm{co^+ex}(n,\mathbb{F}) \leq \frac{2}{3}n.
\]
Moreover, this bound is sharp when $n$ is divisible by $6$.
\end{thm}

\begin{proof}

    Let $H$ be an $n$-vertex $3$-graph with minimum positive co-degree $\delta^+_2(H) > \frac{2}{3}n$.
    Observe that the positive co-degree condition implies that a vertex $x$ with positive degree (i.e., $x$ is contained in a $3$-edge) necessarily has degree greater than $\frac{2}{3}n$.
    
    Let $v_1v_2v_3$ be a $3$-edge. Then $N(v_1,v_2) \cap N(v_3)$ is non-empty, so there is a vertex $v_4 \in N(v_1,v_2) \cap N(v_3)$.
    Observe that $N(v_1) \cap N(v_2) \cap N(v_3,v_4)$ is non-empty, so there is a vertex $v_5$ such that $v_3v_4v_5$ is a $3$-edge and the pair $v_5,v_i$ has positive co-degree for $i=1,2$.
    Now $N(v_1,v_5) \cap N(v_2,v_4) \cap N(v_3)$ is non-empty, so there is a vertex $v_6$ such that $v_1v_5v_6$ and $v_2v_4v_6$ are $3$-edges and the pair $v_6,v_3$ has positive co-degree.
    Finally, $N(v_1,v_4) \cap N(v_2,v_5) \cap N(v_3,v_6)$ is non-empty, so there is a vertex $v_7$ such that $v_1v_4v_7,v_2v_5v_7$ and $v_3v_6v_7$ are $3$-edges. These seven $3$-edges form an $\mathbb{F}$.
    
    When $n$ is divisible by $6$, we obtain a matching lower bound by considering the complete balanced $6$-partite $3$-graph. 
\end{proof}

We do not attempt to prove that the complete balanced $6$-partite $3$-graph is uniquely extremal when $n$ is divisible by $6$. When $n$ is not divisible by $6$, the complete balanced $6$-partite $3$-graph still gives a lower bound of at least $4\lfloor \frac{n}{6} \rfloor$, which shows that $\mathrm{co^+ex}(n,\mathbb{F})$ is asymptotic to $\frac{2n}{3}$. However, we do not attempt a more precise result when $6$ does not divide $n$.

\subsection{\texorpdfstring{$K_4$}{text}}\label{k4}

\begin{center}
  \begin{longtable}{ | l ||  c  |  c  |  c  |   }
    \hline
    $K_4$ 
    & 123, 124, 134, 234
    &
\vc{\begin{tikzpicture}\outercycle{5}{4}
\draw (x0) node[unlabeled_vertex]{};\draw (x1) node[unlabeled_vertex]{};\draw (x2) node[unlabeled_vertex]{};\draw (x3) node[unlabeled_vertex]{};
\drawhyperedge{0}{4}
\drawhypervertex{0}{0}
\drawhypervertex{1}{0}
\drawhypervertex{2}{0}
\drawhyperedge{1}{4}
\drawhypervertex{0}{1}
\drawhypervertex{1}{1}
\drawhypervertex{3}{1}
\drawhyperedge{2}{4}
\drawhypervertex{0}{2}
\drawhypervertex{2}{2}
\drawhypervertex{3}{2}
\drawhyperedge{3}{4}
\drawhypervertex{1}{3}
\drawhypervertex{2}{3}
\drawhypervertex{3}{3}
\end{tikzpicture} 
}   
    &
   \Kfourthree
   \\
\hline
  \end{longtable}
\end{center}
The extremal co-degree for 3-graphs excluding a $K_4$ has previously been studied.  It was shown by Czygrinow and Nagle \cite{MR1829685} that the co-degree density $\gamma(K_4)$ is at least $1/2-o(1)$, and they conjecture that $1/2$ is the correct upper bound. The construction achieving this bound is a nice application of the probabilistic method. We provide a different (deterministic) construction of a 3-graph with minimum positive co-degree $n/2 - 1$. For each $n$, our construction gives a larger minimum positive co-degree than the random construction of \cite{MR1829685}, but asymptotically the resulting densities are the same: $1/2$.  This is the one case considered here where we could not improve (at least asymptotically) on the co-degree density.

Finally, while we do not consider larger complete graphs, we note that these have also been studied in the literature. In particular, Lo and Markstr\" om \cite{LoMark} proved that for all $k$ the limit
$$\lim_{n\rightarrow\infty}\frac{1}{n}\textrm{coex}(n,K_k)$$
exists.

Later, Sidorenko\cite{Sidorenkocode} showed  that $$\lim_{k\rightarrow\infty}\lim_{n\rightarrow\infty}\frac{1}{n}\textrm{coex}(n,K_k)\geq 1-o(k^{-1.084}).$$  
Lo and Zhao\cite{AllanZhao} gave a construction showing that  $$\lim_{k\rightarrow\infty}\lim_{n\rightarrow\infty}\frac{1}{n}\textrm{coex}(n,K_k) = 1-\Theta\left(\frac{\ln k}{k^2}\right).$$
The authors of these last three publications in fact prove more general statements for general $r$-graphs. We give the following simple bounds.

\begin{thm}\label{K4}
\[
\frac{n}{2} - 1 \leq \mathrm{co^+ex}(n,K_4)\leq  \frac{2n}{3}.
\]
\end{thm}

\begin{proof}

For the lower bound, we take an $n$-vertex complete one-way bipartite 3-graph with classes $X,Y$. When $n$ is even, put $|X| = |Y| = n/2$. Observe that two vertices in $X$ have co-degree $|Y| = n/2$, two vertices in $Y$ have co-degree zero, and a pair $x \in X, y \in Y$ have co-degree $|X| - 1 = \frac{n}{2} - 1$. When $n$ is odd, we can take $|X| = \lceil n/2 \rceil$ and $|Y| = \lfloor n/2 \rfloor$ to achieve minimum positive co-degree $ \lceil n/2 \rceil  - 1 = \lfloor  n/2 \rfloor > \frac{n}{2} - 1$.

For the upper bound, suppose $\delta_2^+(H) > \frac{2n}{3}$ and let $abc$ be a $3$-edge of $H$. Observe that $|N(a,b) \cap N(a,c) \cap N(b,c)| > 0$, so $H$ contains a $K_4$. 
\end{proof}

Since initial circulation of this manuscript on arXiv, Volec~\cite{volec}  reported that a bound of $\gamma^+(K_4) \leq  0.54296$ can be obtained via Flag Algebras.
An intriguing question raised by Theorem \ref{K4} is whether or not the co-degree and positive co-degree densities are equal when forbidding a $K_4$, i.e., does $\gamma(K_4)=\gamma^+(K_4)$?

\pagebreak

\subsection{\texorpdfstring{$F_{3,3}$}{text}}\label{f33}

\begin{center}
  \begin{longtable}{ | l ||  c  |  c  |  c  |   }
    \hline
$F_{3,3}$
&
 123, 145, 146, 156, 245, 246, 256, 345, 346, 356
 &
\vc{\begin{tikzpicture}\outercycle{7}{6}
\draw (x0) node[unlabeled_vertex]{};\draw (x1) node[unlabeled_vertex]{};\draw (x2) node[unlabeled_vertex]{};\draw (x3) node[unlabeled_vertex]{};\draw (x4) node[unlabeled_vertex]{};\draw (x5) node[unlabeled_vertex]{};
\drawhyperedge{0}{6}
\drawhypervertex{0}{0}
\drawhypervertex{1}{0}
\drawhypervertex{2}{0}
\drawhyperedge{1}{6}
\drawhypervertex{0}{1}
\drawhypervertex{3}{1}
\drawhypervertex{4}{1}
\drawhyperedge{2}{6}
\drawhypervertex{0}{2}
\drawhypervertex{3}{2}
\drawhypervertex{5}{2}
\drawhyperedge{3}{6}
\drawhypervertex{0}{3}
\drawhypervertex{4}{3}
\drawhypervertex{5}{3}
\drawhyperedge{4}{6}
\drawhypervertex{1}{4}
\drawhypervertex{3}{4}
\drawhypervertex{4}{4}
\drawhyperedge{5}{6}
\drawhypervertex{1}{5}
\drawhypervertex{3}{5}
\drawhypervertex{5}{5}
\drawhyperedge{6}{6}
\drawhypervertex{1}{6}
\drawhypervertex{4}{6}
\drawhypervertex{5}{6}
\drawhyperedge{7}{6}
\drawhypervertex{2}{7}
\drawhypervertex{3}{7}
\drawhypervertex{4}{7}
\drawhyperedge{8}{6}
\drawhypervertex{2}{8}
\drawhypervertex{3}{8}
\drawhypervertex{5}{8}
\drawhyperedge{9}{6}
\drawhypervertex{2}{9}
\drawhypervertex{4}{9}
\drawhypervertex{5}{9}
\end{tikzpicture} 
}

 &
\vc{
\begin{tikzpicture}
\draw
(0,-0.7) coordinate(1) node[vtx](a){}
(0,0) coordinate(2) node[vtx](b){}
(0,0.7) coordinate(3) node[vtx](c){}
(1,-0.7) coordinate(4) node[vtx](d){}
(1,0.0) coordinate(5) node[vtx](d){}
(1,0.7) coordinate(6) node[vtx](d){}
;
\draw[hyperedge] (1) to[bend left] (2) to[bend left] (3) to[bend right] (1);
\draw[hyperedge] (1) to[out=0,in=190] (4) to[out=190,in=250] (5) to[out=250,in=0] (1);
\draw[hyperedge] (2) to[out=0,in=140] (4) to[out=140,in=210,looseness=1.5] (5) to[out=210,in=0] (2);
\draw[hyperedge] (3) to[out=-60,in=120,looseness=1.2] (4) to[out=120,in=170,looseness=1.2] (5) to[out=170,in=-60,looseness=1.2] (3);
\draw[hyperedge] (1) to[out=60,in=190] (5) to[out=190,in=250] (6) to[out=250,in=60] (1);
\draw[hyperedge] (2) to[out=0,in=140] (5) to[out=140,in=210,looseness=1.5] (6) to[out=210,in=0] (2);
\draw[hyperedge] (3) to[out=0,in=120,looseness=1.2] (5) to[out=120,in=170,looseness=1.2] (6) to[out=170,in=0,looseness=1.2] (3);
\draw[hyperedge] (1) to[out=10,in=170,looseness=1.5] (4) to[out=170,in=190] (6) to[out=190,in=10] (1);
\draw[hyperedge] (2) to[out=0,in=140] (4) to[out=140,in=210,looseness=1.5] (6) to[out=210,in=0] (2);
\draw[hyperedge] (3) to[out=-10,in=120,looseness=1.2] (4) to[out=120,in=190,looseness=1.2] (6) to[out=190,in=-10,looseness=1.2] (3);
\end{tikzpicture}
}
   \\
   \hline
  \end{longtable}
\end{center}

We begin with a lemma.

\begin{lemma}\label{independent set bound}

Let $H = (V,E)$ be a $3$-graph with an independent set of size $cn$. Then $\delta_{2}^+(H) \leq (1-c)n$. 

\end{lemma}

\begin{proof}

Let $A$ be an independent set of $H$ of size $cn$ and $B = V \setminus A$. If there is no pair $x \in A, y \in B$ such that $x,y$ have positive co-degree, then every vertex in $A$ is isolated, and it is immediate that $H$ has $\delta_2^+(H) \leq (1-c)n - 2 < (1-c)n$. So we may assume that there exists $x \in A, y \in B$ such that $x,y$ have positive co-degree. If $N(x,y) \subset B$, then the co-degree of $x,y$ is at most $|B| - 1 = (1-c)n -1 < (1-c)n$, so $\delta_2^+(H) < (1-c)n$. On the other hand, if there is a $3$-edge $xyz$ with $z \in A$, then $x,z$ have positive co-degree. Since $A$ contains no $3$-edge, we must have $N(x,z) \subseteq B$, so the co-degree of $x,z$ is at most $|B| = (1-c)n$. Thus, $\delta_2^+(H)$ must be at most $(1-c)n$. 
\end{proof}

\begin{thm}
$$3\left \lfloor \frac{n}{5}  \right \rfloor  \leq \mathrm{co^+ex}(n,F_{3,3}) \leq \frac{3}{4}n.$$
\end{thm}

\begin{proof}

For the lower bound, observe that the complete balanced $5$-partite $3$-graph on $n$ vertices contains no $F_{3,3}$ copy and has minimum positive co-degree at least $3 \lfloor \frac{n}{5} \rfloor$.

For the upper bound, let $\varepsilon > 0$ and suppose $H = (V,E)$ is a $3$-graph with minimum positive co-degree $\left( \frac{2}{3} + \varepsilon \right) n$. Let $abc$ be a $3$-edge of $H$. We wish to obtain a lower bound on $|N(a,b) \cap N(b,c) \cap N(a,c)|$. For $x \in V$, we define $f(x)$ to be the number of common neighborhoods in $\{N(a,b), N(b,c), N(a,c)\}$ containing $x$. So $f(x) \in \{0,1,2,3\}$ for all $x$. Notice that 
$$|N(a,b)| + |N(b,c)| + |N(a,c)| = \sum_{x \in V} f(x)$$
and
$$|N(a,b) \cap N(b,c) \cap N(a,c)| = |\{x \in V : f(x) = 3\}|.$$
From the first equation, we know that 
$$\sum_{x \in V} f(x) \geq 3\delta_{2}^+(H) = 3\left( \frac{2}{3} + \varepsilon\right)n = (2 + 3 \varepsilon)n.$$

Observe, if $|\{x \in V : f(x) = 3\}| < 3 \varepsilon n$, then strictly more than $(1-3\varepsilon)n$ vertices $x \in V$ have $f(x) \leq 2$. Therefore,
$$\sum_{x \in V} f(x) < 3(3 \varepsilon n) + 2((1-3\varepsilon)n ) = (2 + 3 \varepsilon)n.
$$
So we must have $|\{x \in V : f(x) = 3\}| \geq 3 \varepsilon n$. Thus, $|N(a,b) \cap N(b,c) \cap N(a,c)| \geq 3 \varepsilon n$. 

Now, suppose further that $H$ is $F_{3,3}$-free. For any $3$-edge $abc$ of $H$, we must have that $N(a,b) \cap N(b,c) \cap N(a,c)$ is independent in order to avoid an $F_{3,3}$. We also know that $|N(a,b) \cap N(b,c) \cap N(a,c)| \geq 3 \varepsilon n$. Thus, $H$ has an independent set of size at least $3 \varepsilon n$, so by Lemma \ref{independent set bound}, we have that $\delta_2^+(H) \leq (1-3\varepsilon)n$. Now, since $\delta_2^+(H) = \left( \frac{2}{3} + \varepsilon \right)n$, we must have 
$$\frac{2}{3} + \varepsilon \leq 1 - 3 \varepsilon,$$
which implies $\varepsilon \leq \frac{1}{12}$. So $\delta_2^+(H) \leq \left( \frac{2}{3} + \frac{1}{12} \right)n = \frac{3}{4}n$.
\end{proof}

Since posting of this manuscript, Balogh and Lidick\'y reported an improved upper bound of $\gamma^+(F_{3,3}) \leq 0.616$ via flag algebras.

\subsection{\texorpdfstring{$C_5$}{text} and \texorpdfstring{$C_5^-$}{text}}\label{c5c5-}

\begin{center}
  \begin{longtable}{ | l ||  c  |  c  |  c  |   }
    \hline
$C_5$
&
123, 234, 345, 145, 125
&
\vc{\begin{tikzpicture}\outercycle{6}{5}
\draw (x0) node[unlabeled_vertex]{};\draw (x1) node[unlabeled_vertex]{};\draw (x2) node[unlabeled_vertex]{};\draw (x3) node[unlabeled_vertex]{};\draw (x4) node[unlabeled_vertex]{};
\drawhyperedge{0}{5}
\drawhypervertex{0}{0}
\drawhypervertex{1}{0}
\drawhypervertex{2}{0}
\drawhyperedge{1}{5}
\drawhypervertex{1}{1}
\drawhypervertex{2}{1}
\drawhypervertex{3}{1}
\drawhyperedge{2}{5}
\drawhypervertex{2}{2}
\drawhypervertex{3}{2}
\drawhypervertex{4}{2}
\drawhyperedge{3}{5}
\drawhypervertex{0}{3}
\drawhypervertex{3}{3}
\drawhypervertex{4}{3}
\drawhyperedge{4}{5}
\drawhypervertex{0}{4}
\drawhypervertex{1}{4}
\drawhypervertex{4}{4}
\end{tikzpicture} 
}
&
\vc{
\begin{tikzpicture}
\path (0,1.2) -- (0,-1.1); 
\draw
\foreach \i in {0,1,...,7}{
(90+72*\i:1) coordinate(\i) node[vtx]{}
}
;
\foreach \i in {0,1,2,3,4}{
\pgfmathtruncatemacro{\j}{\i+1}
\pgfmathtruncatemacro{\k}{\i+2}
\draw[hyperedge] (\i) to[out=230+72*\i,in=270+72*\j,looseness=1] (\j) to[out=270+72*\j,in=310+72*\k,looseness=1] (\k) to[out=310+72*\k,in=230+72*\i,looseness=1] (\i);
}
\end{tikzpicture}
} 
   \\
   \hline
$C_5^-$
&
123, 234, 345, 145
&
\vc{\begin{tikzpicture}\outercycle{6}{5}
\draw (x0) node[unlabeled_vertex]{};\draw (x1) node[unlabeled_vertex]{};\draw (x2) node[unlabeled_vertex]{};\draw (x3) node[unlabeled_vertex]{};\draw (x4) node[unlabeled_vertex]{};
\drawhyperedge{0}{5}
\drawhypervertex{0}{0}
\drawhypervertex{1}{0}
\drawhypervertex{2}{0}
\drawhyperedge{1}{5}
\drawhypervertex{1}{1}
\drawhypervertex{2}{1}
\drawhypervertex{3}{1}
\drawhyperedge{2}{5}
\drawhypervertex{2}{2}
\drawhypervertex{3}{2}
\drawhypervertex{4}{2}
\drawhyperedge{3}{5}
\drawhypervertex{0}{3}
\drawhypervertex{3}{3}
\drawhypervertex{4}{3}
\end{tikzpicture} 
}
&
\vc{
\begin{tikzpicture}
\path (0,1.2) -- (0,-1.1); 
\draw
\foreach \i in {0,1,...,7}{
(90+72*\i:1) coordinate(\i) node[vtx]{}
}
;
\foreach \i in {0,1,2,3}{
\pgfmathtruncatemacro{\j}{\i+1}
\pgfmathtruncatemacro{\k}{\i+2}
\draw[hyperedge] (\i) to[out=230+72*\i,in=270+72*\j,looseness=1] (\j) to[out=270+72*\j,in=310+72*\k,looseness=1] (\k) to[out=310+72*\k,in=230+72*\i,looseness=1] (\i);
}
\end{tikzpicture}
} 
   \\
   \hline
  \end{longtable}
\end{center}

We give elementary bounds for both $C_5$ and $C_5^-$.

\begin{thm}
\[
2\left \lfloor \frac{n}{4} \right \rfloor \leq \mathrm{co^+ex}(n,C_5)\leq  \frac{2}{3}n.
\]
\end{thm}

\begin{proof}
For the lower bound, it is easy to see that $C_5$ is not $4$-partite so it is not contained in the balanced complete $4$-partite $3$-graph. 
For the upper bound, let $H$ be an $n$-vertex $3$-graph with $\delta_2^+(H) > \frac{2}{3}n$. It is easy to find four vertices $v_1,v_2,v_3,v_4$ such that $v_1v_2v_3, v_2v_3v_4 \in E(H)$ and $v_1, v_4$ have positive co-degree. Then observe that
$N(v_1,v_2)\cap N(v_3,v_4)\cap N(v_1,v_4)\ne \emptyset$.  Therefore, there exists a $v_5$ in all three neighborhoods. Thus, $v_4v_5v_1$, $v_5v_1v_2$ and $v_3v_4v_5$ are all $3$-edges in ${H}$. Together with $v_1v_2v_3$ and $v_2v_3v_4$ they form a $C_5$.
\end{proof}

\begin{thm}
\[
\left \lfloor \frac{n}{3} \right \rfloor \leq\mathrm{co^+ex}(n,C_5^-)\leq \frac{1}{2}n.
\]
\end{thm}

\begin{proof}
For the lower bound, consider a triangle blowup, i.e., a balanced complete $3$-partite $3$-graph.  For the upper bound, let $H$ be an $n$-vertex $3$-graph with $\delta_2^+(H) > \frac{1}{2}n$. It is easy to find four vertices $v_1,v_2,v_3,v_4$ such that $v_1v_2v_3, v_2v_3v_4  \in  E(H)$. Observe that  $N(v_1,v_2)\cap N(v_3,v_4)\ne \emptyset$.  Therefore, there exists a vertex $v_5$ in both neighborhoods. Thus, $v_5v_1v_2$ and $v_3v_4v_5$ are also both $3$-edges in ${H}$. Together with $v_1v_2v_3$ and $v_2v_3v_4$ they form a $C_5^-$.
\end{proof}

Since our manuscript was posted to arXiv, Wu~\cite{Wu} has determined the values $\mathrm{co^+ex}(n,C_5)$ and $\mathrm{co^+ex}(n,C_5^-)$ exactly.

\subsection{\texorpdfstring{$J_k$}{text}}\label{jk}
$J_k$ is the 3-graph on $k+1$ vertices with $\binom{k}{2}$ 3-edges such that there is a distinguished vertex contained in every 3-edge, and all other vertex pairs are contained in a 3-edge.

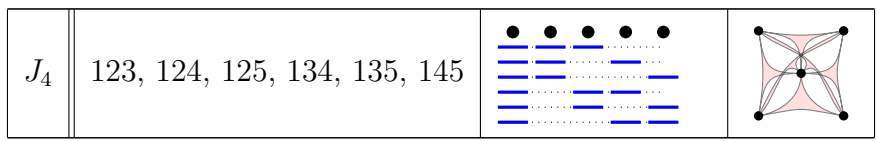
\begin{center}
  \begin{longtable}{ | l ||  c  |  c  |  c  |   }
    \hline
   $J_4$
   &
   123, 124, 125, 134, 135, 145
   &
\vc{\begin{tikzpicture}\outercycle{6}{5}
\draw (x0) node[unlabeled_vertex]{};\draw (x1) node[unlabeled_vertex]{};\draw (x2) node[unlabeled_vertex]{};\draw (x3) node[unlabeled_vertex]{};\draw (x4) node[unlabeled_vertex]{};
\drawhyperedge{0}{5}
\drawhypervertex{0}{0}
\drawhypervertex{1}{0}
\drawhypervertex{2}{0}
\drawhyperedge{1}{5}
\drawhypervertex{0}{1}
\drawhypervertex{1}{1}
\drawhypervertex{3}{1}
\drawhyperedge{2}{5}
\drawhypervertex{0}{2}
\drawhypervertex{1}{2}
\drawhypervertex{4}{2}
\drawhyperedge{3}{5}
\drawhypervertex{0}{3}
\drawhypervertex{2}{3}
\drawhypervertex{3}{3}
\drawhyperedge{4}{5}
\drawhypervertex{0}{4}
\drawhypervertex{2}{4}
\drawhypervertex{4}{4}
\drawhyperedge{5}{5}
\drawhypervertex{0}{5}
\drawhypervertex{3}{5}
\drawhypervertex{4}{5}
\end{tikzpicture} 
}   
   &
\vc{
\begin{tikzpicture}[scale=0.8]
\draw
(0,0) coordinate(1) node[vtx](b){}
(45:1) coordinate(2) node[vtx](c){}
(135:1) coordinate(3) node[vtx](d){}
(225:1) coordinate(4) node[vtx](d){}
(315:1) coordinate(5) node[vtx](d){}
;
\draw[hyperedge] (1) to[out=90,in=190,looseness=1.5] (2) to[out=190,in=-10] (3) to[out=-10,in=90,looseness=1.5] (1);
\draw[hyperedge] (1) to[out=180,in=280,looseness=1.5] (3) to[out=280,in=80] (4) to[out=80,in=180,looseness=1.5] (1);
\draw[hyperedge] (1) to[out=270,in=10] (4) to[out=10,in=170,looseness=1.5] (5) to[out=170,in=270] (1);
\draw[hyperedge] (1) to[out=0,in=100,looseness=1.5] (5) to[out=100,in=260,looseness=1.5] (2) to[out=260,in=0,looseness=1.5] (1);
\draw[hyperedge] (1) to[out=135,in=225,looseness=1.2] (2) to[out=205,in=65,looseness=1.2] (4) to[out=45,in=135,looseness=1.2] (1);
\draw[hyperedge] (1) to[out=45,in=315,looseness=1.2] (3) to[out=335,in=115,looseness=1.2] (5) to[out=135,in=45,looseness=1.2] (1);
\end{tikzpicture}
}   
   \\
   \hline
  \end{longtable}
\end{center}

\begin{thm}
\begin{equation*}
(k-2)\left \lfloor \frac{n}{k} \right \rfloor \leq\mathrm{co^+ex}(n,J_k)\leq \frac{k-2}{k-1}n.
\end{equation*}
\end{thm}
\begin{proof}
For $k\leq 2$, the bounds trivially hold.  For the lower bound when $k>2$, consider a complete $k$-partite graph.  We prove the upper bound by induction on $k$.  Note that the base case, $k=2$ trivially holds.  Assume then that an $n$-vertex hypergraph $H$ has minimum positive co-degree greater than $\frac{k-2}{k-1}n$.  By the induction hypothesis, $H$ contains a $J_{k-1}$.  Let $x$ denote the distinguished vertex of the $J_{k-1}$.  We need to find a vertex $y$ such that $xyz$ form an edge of the graph for each of the $k-1$ vertices $z$ distinct from $x$ in the $J_{k-1}$.  As the co-degree of each pair $z$ and $x$ is greater than $\frac{k-2}{k-1}n$ such a vertex $y$ must exist by the pigeonhole principle.
\end{proof}

For the $J_4$, the lower bound can be improved to $4 \lfloor \frac{n}{7} \rfloor$ with the following construction.

\begin{const} Let $\mathbb{F}$ be the Fano plane (see Section~\ref{fano}).
We take an $n$-vertex balanced blow-up of  the complement of $\mathbb{F}$, that is, a $3$-graph ${H}_7$ with vertex set $V$ of size $n$ partitioned as $V = V_1 \cup \dots \cup V_7$, and $3$-edge set 
$$ E= \{ v_{i}v_{j}v_{k}: \text{ for all } v_i \in V_i, v_{j} \in V_{j}, v_k \in V_k \text{ such that } i<j<k \text{ and } ijk \not \in  \mathbb{F}.\}$$
\end{const}

It is easy to check that $J_4$ is not a subhypergraph of the complement of $\mathbb{F}$. 
Moreover, every pair of vertices in $J_4$ has non-zero co-degree. Therefore, the blow-up $H_7$ is $J_4$-free.
A pair of vertices in the same class $V_i$ have co-degree $0$. Now let $x,y$ be vertices in distinct classes $V_i,V_j$. There is a unique edge $ijk$ in the Fano plane corresponding to classes $V_i$ and $V_j$. Therefore, there is an edge $xyz$ for each vertex $z$ in a class other than $V_i,V_j,V_k$. Thus, $x,y$ has co-degree at least  $4 \lfloor \frac{n}{7} \rfloor$.

Therefore, $\gamma^+(J_4) \geq 4/7 > 0.57$.
Since initial circulation of these results, Balogh and Lidick\'y~\cite{BL} 
reported an improved upper bound of $\gamma^+(J_4) \leq 0.58$ via flag algebras.

\section{General results}\label{sec-secondary}

In this section, we collect results of a more general flavor. Our first goal is to establish the existence of the limit 
$$\gamma^+(F) = \lim_{n \rightarrow \infty} \frac{\mathrm{co^+ex}(n,F)}{n}.$$

As previously mentioned, the existence of this limit can be proved by adapting the argument of Mubayi and Zhao~\cite{Mubayicode} which establishes the existence of the limit
$$\gamma(F) = \lim_{n \rightarrow \infty} \frac{\mathrm{coex}(n,F)}{n}.$$
We refer the reader to \cite{P} for a rigorous treatment. We shall give a different proof of the existence of $\lim_{n \rightarrow \infty} \frac{\mathrm{co^+ex}(n,F)}{n}$, which highlights some of the nice aspects of working with positive co-degree; the key argument which allows our approach will also yield simple proofs of general supersaturation and invariance of co-degree density under the blow-up operation. In fact, we shall prove supersaturation and blow-up invariance before proving the existence of $\gamma^+(F)$.

The key step in these proofs will be an application of the hypergraph removal lemma, which we state below (see, for example, \cite{ConlonFox}).

\begin{lemma}\label{hypergraph removal ch4} 
Let $F$ be an $r$-graph and fix $\varepsilon > 0$. There exists $\delta > 0$ such that if $H$ is an $n$-vertex $r$-graph containing at most $\delta n^{|V(F)|}$ copies of $F$, then there exists $E' \subset E(H)$ such that $|E'| \leq \varepsilon n^r$ and $H - E'$ is $F$-free.

\end{lemma}

Roughly speaking, given an $r$-graph $H$ with few copies of $F$, we can remove these copies via few $r$-edge deletions to obtain an $F$-free subhypergraph $H_1$ of $H$. However, it is unclear that $\delta_{r-1}^+(H_1)$ should be close to $\delta_{r-1}^+(H)$. In fact, it may be the case that $\delta_{r-1}^+(H_1)$ is much smaller than $\delta_{r-1}^+(H)$. The following lemma, from which a number of general theorems will quickly follow, shows that we can ``clean up'' $H_1$ to obtain a subhypergraph $H_2$ with $\delta_{r-1}^+(H_2)$ near to $\delta_{r-1}^+(H)$. Note that such a ``clean-up'' process would not work in the ordinary co-degree setting, since the procedure relies upon deleting $r$-edges so that sets whose co-degree is small in $H_1$ will have co-degree $0$ in $H_2$.

Throughout this section we will use $d_{r-1}(S)$ to denote the co-degree of an $(r-1)$-set $S$, i.e., the number of $r$-edges containing $S$.

\begin{lemma} \label{cleanup lemma}

Let $H$ be an $n$-vertex $r$-graph and fix $0< \varepsilon <1$ small enough that 
\[(r+1)! \sqrt[2^{r-1}]{\varepsilon}n^r < |E(H)|.\] Let $H_1$ be a subhypergraph of $H$ obtained by the deletion of at most $\varepsilon n^r$ $r$-edges. Then $H_1$ has a subhypergraph $H_2$ 
with $\delta_{r-1}^+(H_2) \geq \delta_{r-1}^+(H) - 2^r r! \sqrt[2^{r-1}]{\varepsilon}n$.

\end{lemma}

\begin{proof}

Let $E'$ be the set of $r$-edges which we delete from $H$ to obtain $H_1$. For $S \in \binom{V(H)}{r-1}$, let $f_{r-1}(S)$ be the amount by which $d_{r-1}(S)$ drops after the deletion of $E'$. Each deleted $r$-edge drops $d_{r-1}(S)$ by $1$ for $r$ different $(r-1)$-sets $S$, so 
$$\sum_{S} f_{r-1}(S) \leq r \varepsilon n^r.$$

Thus, at most $r \sqrt{\varepsilon}n^{r-1}$ sets $S$ have $f(S) \geq \sqrt{\varepsilon} n$. We shall ultimately delete all edges of $H_1$ which contain an $(r-1)$-set $S$ with $f_{r-1}(S) \geq \sqrt{\varepsilon} n$. However, this deletion alone may be insufficient to guarantee the desired minimum positive co-degree. We make the following definition.  We say that an $(r-1)$-set $S_{r-1}$ is $(r-1)$-\textit{bad} if $S_{r-1}$ is contained in at least $\sqrt{\varepsilon} n$ edges of $E'$, i.e., if $f_{r-1}(S) \geq \sqrt{\varepsilon} n$. For $1 \leq i \leq r-2$, we say that an $i$-set $S_i$ is $i$-\textit{bad} if $S_i$ is contained in at least $\sqrt[2^{r-i}]{\varepsilon}n$ sets which are $(i+1)$-bad. We shall in fact delete all $r$-edges of $H_1$ which contain any $i$-bad set, for any $i \in \{1, \dots, r-1\}$.

We bound the number of $i$-bad sets in $H_1$. First, we consider $(r-2)$-bad sets. For a set $S_{r-2} \in \binom{V(H_1)}{r-2}$, let $f_{r-2}(S_{r-2})$ be the number of $(r-1)$-bad sets containing $S_{r-2}$. Since $H_1$ contains at most $r \sqrt{\varepsilon} n^{r-1}$ $(r-1)$-bad sets, and since each $(r-1)$-bad set contains $r-1$ sets of size $r-2$, we have
\[ \sum_{S_{r-2}} f_{r-2}(S_{r-2}) \leq (r-1) \cdot r \sqrt{\varepsilon} n^{r-1}.\]
Thus, at most $r(r-1) \sqrt[4]{\varepsilon} n^{r-2}$ sets $S_{r-2}$ have $f_{r-2}(S_{r-2}) \geq \sqrt[4]{\varepsilon}n$, i.e., at most $\frac{r!}{(r-2)!}\sqrt[4]{\varepsilon}n^{r-2}$ sets of size $r-2$ are $(r-2)$-bad. Proceeding inductively, at most $\frac{r!}{i!} \sqrt[2^{r-i}]{\varepsilon} n^i$ $i$-sets are $i$-bad.

Thus, for any $i$, there are few $i$-bad sets in $V(H_1)$. Now, we delete all $r$-edges of $H_1$ which contain any $i$-bad set. Let $H_2$ be the resulting subhypergraph of $H_1$. We claim that $H_2$ has the desired minimum positive co-degree. 
Indeed, suppose $T$ is an $(r-1)$-set of vertices with $d_{r-1}(T) > 0$ in $H_2$. Thus, no subset of $T$ is bad. We now estimate the number of $r$-edges containing $T$ which have been deleted.

Suppose $T \subset e \in E(H) \setminus E(H_2)$. Thus, either $e \in E'$ or there is some bad $i$-set $S_i \subset e$. Since $T$ itself is not $(r-1)$-bad, $T$ is contained in at most $\sqrt{\varepsilon} n$ $r$-edges of $E'$. If $e$ is not in $E'$, but contains an $i$-bad set $S_i$, then $|S_i \cap T| = i-1$. Let $T_{i-1}= S_i \cap T$ be the set of $i-1$ vertices of $T$ which are distinguished by $e$. If $T_{i-1} = \emptyset$, then $S_i$ is a $1$-bad set. There are at most $r!\sqrt[2^{r-1}]{\varepsilon}n$ $1$-bad sets. Now, suppose $T_{i-1} \neq \emptyset$. Since $T_{i-1}$ is not $(i-1)$-bad, $T_{i-1}$ is contained in fewer than $\sqrt[2^{r-i+1}]{\varepsilon} n$ $i$-bad sets. Thus, each subset of $T$ of size $i-1$ corresponds to fewer than $\sqrt[2^{r-i+1}]{\varepsilon} n$ $r$-edges containing $T$ which are deleted from $H_1$. Therefore, the described deletions drop the co-degree of $T$ by at most
\[ \sqrt{\varepsilon}n + r!\sqrt[2^{r-1}]{\varepsilon}n + \sum_{i=1}^{r-1} \binom{r-1}{i} \sqrt[2^{r-i+1}]{\varepsilon}n \leq 2^r r! \sqrt[2^{r-1}]{\varepsilon}n.\]

Thus, any set $T$ which maintains positive co-degree in $H_2$ satisfies the desired minimum positive co-degree condition. It remains to show that $H_2$ is not empty. We estimate the total number of $r$-edges deleted from $H$ to obtain $H_2$. At most $\varepsilon n$ $r$-edges are deleted to obtain $H_1$. From $H_1$, we delete all $r$-edges which contain any $i$-bad set, for any $1 \leq i \leq r-1$. A fixed $i$-bad set is in at most $n^{r-i}$ $r$-edges, so, using previous bounds on the number of $i$-bad sets, we delete at most 

\[ \sum_{i=1}^{r-1} n^{r-i} \frac{r!}{i!} \sqrt[2^{r-i}]{\varepsilon} n^i \]
$r$-edges from $H_1$ to obtain $H_2$. Thus, the total number of $r$-edge deletions required to obtain $H_2$ is at most 
\[\varepsilon n^r + \sum_{i=1}^{r-1} n^{r-i} \frac{r!}{i!} \sqrt[2^{r-i}]{\varepsilon} n^i < (r+1)! \sqrt[2^{r-1}]{\varepsilon}n^r\]
By hypothesis, $|E(H)| > (r+1)! \sqrt[2^{r-1}]{\varepsilon}n^r$, so $H_2$ is not empty.\qedhere

\end{proof}

Note that in general, Lemma~\ref{cleanup lemma} will only be useful for $r$-graphs with order $n^r$ $r$-edges, since we must be able to pick $\varepsilon$ with $(r+1)! \sqrt[2^{r-1}]{\varepsilon}n^r < |E(H)|$. However, recall that Lemma~\ref{edge approx} tells us that if $\delta_{r-1}^+(H)$ is linear, then $|E(H)| = \Theta(n^r)$. Thus, in situations where we cannot apply Lemma~\ref{cleanup lemma}, we will have $\delta_{r-1}^{+}(H) = o(n)$, which will be enough for us to work with.

Our first application of Lemma~\ref{cleanup lemma} is a general supersaturation result.

\begin{thm}\label{supersat}

Let $F$ be an $r$ graph and fix $\varepsilon > 0$. There exists $\delta > 0$ such that, if $H$ is an $n$-vertex $r$-graph with 
$$\delta_{r-1}^+(H) >\mathrm{co^+ex}(F) + \varepsilon n,$$
then $H$ contains at least $\delta n^{|V(F)|}$ copies of $F$. 
\end{thm}

\begin{proof}

Since $\delta_{r-1}^+(H) \geq \varepsilon n$, we know $|E(H)| = \Theta(n^r)$ by Lemma~\ref{edge approx}. Fix $\alpha > 0$ such that $(r+1)! \sqrt[2^{r-1}]{\alpha}n^r < |E(H)|$ and $2^r r! \sqrt[2^{r-1}]{\alpha} < \varepsilon$, and apply Lemma~\ref{hypergraph removal ch4} with $\delta = \delta(\alpha)$. Thus, if $H$ contains at most $\delta n^{|V(F)|}$ copies of $F$, then $H$ can be made $F$-free with at most $\alpha n^r$ $r$-edge deletions. Let $H_1$ be the resulting subhypergraph of $H$. By Lemma~\ref{cleanup lemma} and the choice of $\alpha$, $H_1$ contains a subhypergraph $H_2$ with $\delta_{r-1}^+(H_2) >  \delta_{r-1}^+(H) - \varepsilon n > \mathrm{co^+ex}(n,F)$, a contradiction, as $H_2$ is $F$-free. We conclude that $H$ contains more than $\delta n^{|V(F)|}$ copies of $F$.
\end{proof}

As a standard consequence of supersaturation (see, e.g.~\cite{Keevashsurvey}), we have blow-up invariance for  $\mathrm{co^+ex}(n,F)$. First we set some additional notation. Given an $r$-graph $H$ and an integer $t$, the \textit{t-blow-up} $H[t]$ is the $r$-graph obtained by replacing each vertex $v_i \in V(H)$ with a strongly independent class $V_i$ of $t$ vertices. A set of $r$ vertices in $H[t]$ spans an $r$-edge if and only if the vertices' classes correspond to the $r$ vertices of an $r$-edge in $H$. We will also sometimes wish to consider blow-ups whose classes are not exactly equal in size. We shall denote by $H[n/v(H)]$ a blow-up of $H$ with $n$ vertices and vertex classes of size $\lfloor n/v(H) \rfloor$ or $\lceil n/v(H) \rceil$, called a \textit{balanced} blow-up of $H$. Note that there is not always a unique such blow-up. However, in the one instance where we work with balanced blow-ups which are not of the form $H[t]$, our argument will apply to any balanced blow-up of appropriate size; thus, we are happy to leave the notation somewhat ambiguous. Note that when $v(H)$ divides $n$, there is a unique balanced blow-up of $H$ with vertex classes of size $n/v(H)$, i.e., if $n/v(H) = t$, then $H[n/v(H)] = H[t]$. 

\begin{cor} \label{blowup invariance}

Let $F$ be an $r$-graph and $t$ a positive integer. Then
$$\mathrm{co^+ex}(n,F) \leq \mathrm{co^+ex}(n,F[t]) \leq \mathrm{co^+ex}(n,F) + o(n).$$
\end{cor}

Now, we are nearly ready to prove that $\gamma^+(F) = \lim_{n \rightarrow \infty} \frac{\mathrm{co^+ex}(n,F)}{n}$ exists. The broad strategy will be to show that $c_n := \frac{\mathrm{co^+ex}(n,F)}{n}$ becomes arbitrarily close to the least upper bound $\ell$ on $\{c_n\}$ by constructing (for $n$ large enough) $n$-vertex, $F$-free $r$-graphs with minimum positive co-degree nearly $\ell n$. We obtain these $r$-graphs by finding one ``good'' construction and taking balanced blow-ups. However, for general $H,F,$ and $t$, it is possible that $F \not\subset H$ but $F \subset H[t]$, so we shall have to be somewhat careful.

Given two $r$-graphs, $H$ and $F$, we shall say that $H$ \textit{blowup-contains} $F$ if there is some $t$ for which $F \subseteq H[t]$. For a fixed $r$-graph $F$ on $f$ vertices, we define $\mathcal{B}(F)$ to be the family of $r$-graphs on at most $f$ vertices which blowup-contain $F$. Note that $\mathcal{B}(F)$ always contains $F$; depending upon the structure of $F$, $\mathcal{B}(F)$ may contain other $r$-graphs as well. However, since its members have bounded size, $\mathcal{B}(F)$ will always be a finite set. The standard proof of Corollary~\ref{blowup invariance} can be easily adapted to the following result.

\begin{cor} \label{blowup containment corollary}
For any $r$-graph $F$,
$$\mathrm{co^+ex}(n, \mathcal{B}(F)) \leq \mathrm{co^+ex}(n,F) \leq \mathrm{co^+ex}(n,\mathcal{B}(F)) + o(n).$$
\end{cor}

\begin{proof}
Fix $\varepsilon> 0$ and let $F$ be an $r$-graph on $f$ vertices. We wish to show that, for $n$ large enough, any $n$-vertex $r$-graph $H$ with $\delta_{r-1}^+(H) \geq  \mathrm{co^+ex}(n,\mathcal{B}(F)) + \varepsilon n$ contains a copy of $F$. Put $k:=|\mathcal{B}(F)|$, and label the members of $\mathcal{B}(F)$ as $F_1, \dots, F_k$. Observe, for any $F_i \in \mathcal{B}(F)$, we have $F \subset F_i[f]$
, so it suffices to show that $H$ contains a large enough blow-up of some $F_i$. This fact follows by essentially the same proof as Corollary~\ref{blowup invariance}; fixing $\alpha$ as in the proof of Theorem~\ref{supersat}, we apply the Hypergraph Removal Lemma $k$ times to find $\delta_1, \dots, \delta_k$ such that if $H$ contains at most $\delta_i n^{|V(F_i)|}$ copies of $F_i$, then all copies of $F_i$ can be destroyed by the deletion of at most $\frac{\alpha}{k}n^{r}$ $r$-edges. If $H$ in fact contains at most $\delta_i n^{|V(F_i)|}$ copies of $F_i$ for all $i$, then $H$ can be made $\mathcal{B}(F)$-free by deleting at most $\alpha n^r$ $r$-edges. Thus we can apply Lemma \ref{cleanup lemma} obtain a $\mathcal{B}(F)$-free subhypergraph $H'$ of $H$ with $\delta_{r-1}^+(H') > \delta_{r-1}^+(H) - \varepsilon n \geq \mathrm{co^+ex}(n,\mathcal{B}(H))$ for a contradiction.
Thus, for some $i$, $H$ contains at least $\delta_i n^{|V(F_i)|}$ copies of $F_i$. By a standard hypergraph Ramsey argument, this implies that $H$ contains a copy of $F_i[f]$ for $n$ large enough.
\end{proof}

Now, observe that if $H$ is $\mathcal{B}(F)$-free, then any blow-up of $H$ must be $\mathcal{B}(F)$-free. By Corollary~\ref{blowup containment corollary}, in order to prove that the limit
$$\gamma^+(F) = \lim_{n \rightarrow \infty} \frac{\mathrm{co^+ex}(n,F)}{n}$$
exists, it suffices to show the existence of the limit
$$\gamma^+(\mathcal{B}(F)) = \lim_{n \rightarrow \infty} \frac{\mathrm{co^+ex}(n,\mathcal{B}(F))}{n}.$$
We do so now.

 \begin{thm}
For any $r$-graph $F$, the limit 
$$\gamma^+(F) = \lim_{n \rightarrow \infty} \frac{\mathrm{co^+ex}(n,F)}{n}$$
exists. Moreover, for any $\varepsilon > 0$, there exists an $F$-free $r$-graph $H = H(F,\varepsilon)$ such that every blow-up of $H$ is $F$-free, and for all $n$ sufficiently large,
$$\frac{\delta_{r-1}^+(H[n/v(H)])}{n} > \gamma^+(F) - \varepsilon.$$ 

\end{thm}

\begin{proof} 
As described above, we instead will show the existence of 
$$\underset{n \rightarrow \infty}{\lim} \frac{\mathrm{co^+ex}(n,\mathcal{B}(F))}{n}.$$ 
Let $c_n := \frac{\mathrm{co^+ex}(n,\mathcal{B}(F))}{n}$. Note that adding an isolated vertex to an $F$-free $r$-graph does not change the minimum positive co-degree, so $\mathrm{co^+ex}(n,\mathcal{B}(F)) \leq \mathrm{co^+ex}(n+1,\mathcal{B}(F))$. On the other hand, removing any vertex from a $\mathcal{B}(F)$-free $r$-graph drops the minimum positive co-degree by at most $1$, so $\mathrm{co^+ex}(n+1,\mathcal{B}(F))  \leq \mathrm{co^+ex}(n,\mathcal{B}(F)) + 1$.

From here we get 
$$\frac{\mathrm{co^+ex}(n,\mathcal{B}(F)) }{n+1} \leq c_{n+1} \leq \frac{\mathrm{co^+ex}(n,\mathcal{B}(F)) + 1}{n+1}.$$

Observe that
$$
\frac{\mathrm{co^+ex}(n,\mathcal{B}(F))}{n+1}  = \frac{\mathrm{co^+ex}(n,\mathcal{B}(F))}{n+1} \frac{n}{n} = c_n \frac{n}{n+1} = c_n \left(1-\frac{1}{n+1} \right).
$$
Similarly,
\begin{align*}
\frac{\mathrm{co^+ex}(n,\mathcal{B}(F))+1}{n+1}  = \frac{\mathrm{co^+ex}(n,\mathcal{B}(F))+1}{n+1} \frac{n}{n} 
= \left(c_n + \frac{1}{n}\right) \frac{n}{n+1} = \left(c_n + \frac{1}{n}\right) \left(1-\frac{1}{n+1} \right).     
\end{align*}

Therefore,
$$ c_n \left(1-\frac{1}{n+1} \right) \leq c_{n+1} \leq \left(c_n + \frac{1}{n}\right) \left(1-\frac{1}{n+1} \right).$$
Since $c_n \leq 1$ for all $n$, observe that 
$$ c_n - \frac{1}{n} \leq c_n \left(1-\frac{1}{n+1} \right)$$
and clearly
$$\left(c_n + \frac{1}{n}\right) \left(1-\frac{1}{n+1} \right) \leq c_n + \frac{1}{n}.$$
So we conclude that 
$$c_n - \frac{1}{n} \leq c_{n+1} \leq c_n + \frac{1}{n}.$$
Thus,
$$|c_{n+1} - c_n | \leq \frac{1}{n},$$
and it follows that for $k \geq 0$, 
$$|c_{n+k} - c_n| \leq \sum_{i = 0}^{k-1} \frac{1}{n + i}.$$

As $0 \leq c_n \leq 1$ for all $n$, there is a least upper bound $\ell$ on the sequence $\{c_n\}$.  We shall prove that $\underset{n \rightarrow \infty}{\lim} c_n = \ell$. Fix $\varepsilon > 0$. Since $\ell$ is the least upper bound for $\{c_n\}$, there exists some $k$ such that $c_k > \ell - \frac{\varepsilon}{2}$. Let $ {H}_k$ be a $k$-vertex extremal $r$-graph for $\mathcal{B}(F)$.

Now, observe that if $ {H}$ is a $\mathcal{B}(F)$-free $r$-graph, then any blow-up of ${H}$ is also $\mathcal{B}(F)$-free. Indeed, suppose that some blow-up $H[t]$ of $H$ contains a copy of some $F_1 \in \mathcal{B}(F)$, and let $V_1, \dots, V_k$ be the classes of $H[t]$ which intersect this copy. These classes $V_1, \dots, V_k$ correspond to vertices $v_1, \dots v_k \in V(H)$. Let $F_2 = H[v_1, \dots, v_k]$, the subhypergraph of $H$ induced on $v_1, \dots, v_k$. We claim that $F_2 \in \mathcal{B}(F)$, a contradiction to the assumption that $H$ is $\mathcal{B}(F)$-free. Indeed, we have $v(F_2) \leq v(F_1) \leq v(F)$, so we will have $F_2 \in \mathcal{B}(F)$ as long as $F$ is contained in some blow-up of $F_2$. We know that $F_1 \subseteq F_2[t]$, and since $F_1 \in \mathcal{B}(F)$, there exists $t'$ such that $F \subseteq F_1[t']$. Thus, $F \subseteq F_1[t'] \subseteq F_2[tt']$, so indeed, $F_2 \in \mathcal{B}(F)$.

Now in particular, for any integer $m \geq 1$, the blow-up ${H}_{k}[m]$ is $F$-free. Observe that 
$$\delta_{r-1}^+( {H}_{k}[m]) = m \cdot \delta_{r-1}^+( {H}_k),$$
which shows that $c_{mk} \geq c_{k} > \ell - \frac{\varepsilon}{2}$ for all integers $m \geq 1$. 

Now, choose $N$ such that 
$k \leq N$ and $\sum_{i = 0}^{k- 2} \frac{1}{N+ i} < \frac{\varepsilon}{2}$.
Let $n \geq N$, and choose an integer $m$ so that $mk$ is the smallest multiple of $k$ with $n \leq mk$. Since $|mk - n| \leq k-1$, we know that 
$$|c_{mk} - c_{n}| \leq \sum_{i=0}^{k - 2} \frac{1}{n + i} \leq \sum_{i=0}^{k - 2} \frac{1}{N + i} < \frac{\varepsilon}{2}.$$
Since $\ell - \frac{\varepsilon}{2}< c_{mk}$, we must have $\ell - \varepsilon < c_n$. Thus, $\{c_n\}$ converges to $\ell$.

For the second part of the theorem statement, we must show that for any $\varepsilon > 0$, there exists some $H$ such that all sufficiently large balanced blow-ups $H[n/v(H)]$ are $F$-free and have $\frac{\delta_{r-1}^+(H[n/v(H)])}{n} > \gamma^+(F) - \varepsilon$. As above, we choose $k$  such that $c_k > \gamma^+(F) - \frac{\varepsilon}{2}$ and let $ {H}_k$ be a $k$-vertex extremal $r$-graph for $\mathcal{B}(F)$. We know that $\frac{\delta_{r-1}^+(H_k[q])}{qk} = \frac{\delta_{r-1}^+(H_k)}{k} > \gamma^+(F) - \varepsilon$ for all positive integers $q$, so we simply must address blow-ups $H_k[n/k]$ where $k$ does not divide $n$.  Let $n = qk + r$ for integers $q,r$ with $0 \leq r < k$. We have 
\begin{align*}
\frac{\delta_{r-1}^+(H_{k}[n/k])}{qk+r}& \geq \frac{\delta_{r-1}^+(H_{k}[q])}{qk+r} = \frac{\delta_{r-1}^+(H_{k}[q])}{qk} - \frac{r}{qk + r} \frac{\delta_{r-1}^+(H_k[q])}{qk} \\
& > \left( 1 - \frac{1}{q}\right) \frac{\delta_{r-1}^+(H_k[q])}{qk} = \left( 1 - \frac{1}{q}\right) \frac{\delta_{r-1}^+(H_k)}{k} > \left( 1 - \frac{1}{q}\right)\left( \gamma^+(F) - \frac{\varepsilon}{2}\right).  
\end{align*}

If $n$ is large enough that $q \geq \frac{2}{\varepsilon}$, we have $\left( 1 - \frac{1}{q}\right)\left( \gamma^+(F) - \frac{\varepsilon}{2}\right) > \gamma^+(F) - \varepsilon$.
Thus, for $n$ sufficiently large, the balanced blow-up $H_k[n/k]$ has $\delta_{r-1}^+(H_k[n/k])$ above the desired threshold.
\end{proof}

We now turn our attention to $3$-partite $3$-graphs. We start by considering $K_{2,2,2}$.

\begin{thm}\label{thm:K222}  $\mathrm{co^+ex}(n,K_{2,2,2}) = O\left(n^{5/6}\right)$.
\end{thm}

\begin{proof}
Let $H$ be an $n$-vertex 3-graph with maximal minimum positive co-degree such that $H$ contains no $K_{2,2,2}$. Note that we may assume that fewer than ${n}/{2}$ vertices of $H$ are isolated (if not, we can replace a subset of the isolated vertices of $H$ with a copy of the non-trivial components of $H$, thereby reducing the number of isolated vertices without changing the positive co-degree). Let $\delta_2^+ = \delta_2^+(H)$ be the minimum positive co-degree of $H$. 

Let $T$ denote the number of ordered pairs  $(\{x,y\},\{z_1,z_2\})$ such that $x,y,z_i$ form a $3$-edge of $H$ for $i=1,2$.  On the one hand if $z_1$ and $z_2$ are fixed, then the admissible pairs $\{x,y\}$ form a $C_4$-free 2-graph.  This gives the upper bound
\begin{equation*}
T \leq \binom{n}{2}cn^{3/2}.
\end{equation*}
On the other hand, let $E^+$ be the set of (unordered) pairs $\{x,y\}$ with positive co-degree. Then 
\begin{equation*}\label{eq:lowerbound}
|E^+| \binom{\delta_2^+}{2}\leq \sum_{\{x,y\} \in E^+}\binom{d(x,y)}{2} = T.
\end{equation*}

Since at most $n/2$ vertices of $H$ are isolated, we have that $| {E}^+| \geq \frac{n}{4}\delta^+_2$.  Combining these estimates for $T$ gives
\begin{equation*}
\left({\delta^+_2}\right)^3\leq O\left(n^{5/2}\right).
\end{equation*} 
\end{proof}

Thus, unlike previous examples, $\mathrm{co^+ex}(n,K_{2,2,2})$ is sub-linear. The next construction shows that, at least, $\mathrm{co^+ex}(n,K_{2,2,2})$ is not constant.

\begin{thm}\label{K222 lower bound}

$\mathrm{co^+ex}(n,K_{2,2,2}) = \Omega(n^{1/2})$.

\end{thm}

\begin{proof}

Partition $n$ vertices into three classes, $X,Y,Z$, so that $|Y| = |Z| \geq cn$ for some $c > 0$ and $|X| = \Omega(n^{1/2})$. Let $G$ be a $C_4$-free bipartite graph with classes $Y,Z$, with minimum degree as large as possible. Now, let $ {H}$ be the $3$-graph with $3$-edges $xyz$ such that $x \in X, y \in Y, z \in Z$, and $yz$ is an edge in the graph $G$. 

The fact that ${H}$ is $K_{2,2,2}$-free follows directly from the condition that $G$ is $C_4$-free. Indeed, since $ {H}$ is $3$-partite, if it contains a $K_{2,2,2}$, then this subgraph has two vertices in each of $X,Y,Z$. Let $x_1 ,x_2 \in X$,  $y_1 ,y_2 \in X$ and $z_1,z_2 \in Z$. If these six vertices induce a $K_{2,2,2}$, then $x_iy_jz_k$ is a $3$-edge for all choices of $(i,j,k)$. But this means that $y_jz_k$ is an edge in $G$ for all choices of $(j,k)$, i.e., a $G$ contains a $C_4$, a contradiction.

Now, we consider co-degrees in ${H}$. Let $x \in X$, $y \in Y$, and $z \in Z$. Observe that $y,z$ have co-degree either 0 (if $yz$ is not an edge in $G$) or $|X|$, which we assume to be $\Omega(n^{1/2})$. Next, $x,y$ have co-degree equal to the degree of $y$ in $G$. We chose $G$ so that $\delta(G)$ is maximized, so we know that $\delta(G) = \Omega(n^{1/2})$ by a construction of Erd\H os, R\'enyi and S\'os~\cite{ERS}. So the co-degree of $x,y$ (and, analogously, $x,z$) is also $\Omega(n^{1/2})$. 
\end{proof}

We do not make a further effort to determine the order of $\mathrm{co^+ex}(n,K_{2,2,2})$, but a resolution of this question remains interesting.

We next consider 3-partite 3-graphs more broadly. Unlike hypergraphs considered in the previous section, we have seen that $\mathrm{co^+ex}(n,K_{2,2,2})=o(n)$, i.e., $\gamma^+(K_{2,2,2}) = 0$. A natural goal is to characterize those $3$-graphs ${F}$ for which $\gamma^+(F) = 0$.

\begin{thm}\label{thm:degeneracy}
Let $F$ be a 3-graph. If $F$ is 3-partite, then $\mathrm{co^+ex}(n,F)\leq  O(n^{1-\varepsilon})$ for $\varepsilon = \varepsilon(F)$. Otherwise, $\mathrm{co^+ex}(n,F) =  \Theta(n)$.
\end{thm}

\begin{proof}
If $F$ is not 3-partite, then it is not contained in a complete balanced $3$-partite $n$-vertex 3-graph, so $\mathrm{co^+ex}(n,F) =  \Theta(n)$.

Now suppose that $F$ is 3-partite and
choose $j\leq k \leq l$ minimal such that $F \subseteq K_{j,k,l}$. We have $\mathrm{co^+ex}(n,F) \leq \mathrm{co^+ex}(n,K_{j,k,l})$. Let $H$ be an $n$-vertex $3$-graph with maximal minimum positive co-degree such that $H$ is $K_{j,k,l}$-free.

To bound $\mathrm{co^+ex}(n,K_{j,k,l})$, we count pairs $(\{x,y\},\{z_1,z_2,\ldots,z_l\})$ such that $x,y,z_i$ form a $3$-edge of $H$ for each $i\leq l$.  Denote the total number of such pairs by $T$. The proof follows by a similar counting argument used in the proof of Theorem~\ref{thm:K222}.
\end{proof} 

For a short proof of a slightly weaker result, it was shown in \cite{Er1} that for $3$-partite $F$, the Tur\'{a}n number $\ex_3(n, F)$ is $O(n^{3-c})$ for $c = c(F)$. Thus, by the contrapositive of Lemma~\ref{edge approx}, it is immediate that any $n$-vertex $F$-free $3$-graph has minimum positive co-degree $o(n)$. Another short proof of this slightly weaker statement follows from Corollary \ref{blowup containment corollary}. Note that if $F$ is $3$-partite, then $\mathcal{B}(F)$ contains the $3$-graph consisting of a single $3$-edge. Thus, $\mathrm{co^+ex}(n, \mathcal{B}(F)) = 0$, so $\mathrm{co^+ex}(n,F) \leq 0 + o(n) = o(n)$. In fact, these arguments generalize to all uniformities:

\begin{cor}\label{general degeneracy}
Let $F$ be an $r$-graph. If $F$ is $r$-partite, then $\gamma^+(F) = 0$. 
\end{cor}

In light of Corollary \ref{general degeneracy}, we can now quickly show that the positive co-degree density of $r$-graphs jumps from $0$ to $1/r$.

\begin{cor}\label{codegree jump}
Let $F$ be an $r$-graph such that $\gamma^+(F)>0$. Then $\gamma^+(F) \geq 1/r$.
\end{cor}

\begin{proof}
If $F$ is such an $r$-graph, then
by Corollary \ref{general degeneracy}, $F$ cannot be $r$-partite.  Therefore for all $n$, $F$ is not a subgraph of the complete balanced $r$-partite $r$-graph on $n$ vertices. This implies that $\mathrm{co^+ex}(n,F) \geq \lfloor n/r \rfloor$ for all $n$, which completes the proof.
\end{proof}

Note that it also follows immediately that for any family $\mathcal{F}$ of forbidden $r$-graphs either has $\gamma^+(\mathcal{F}) = 0$ or $\gamma^+(\mathcal{F}) \geq 1/r$.  Thus, we say that positive co-degree density ``jumps" from 0 to $1/r$. In a forthcoming manuscript, we develop additional techniques which demonstrate at least one more jump in positive co-degree density for $3$-graphs; thus, the situation described in Corollary~\ref{codegree jump} is not an isolated phenomenon.

\section*{Acknowledgements}

The authors are grateful to J\'ozsef Balogh, Felix Christian Clemen and Bernard Lidick\'y for  permission to use their figures from~\cite{l2norm}.
We thank Sam Spiro for suggesting the first alternative proof of Theorem~\ref{thm:degeneracy} and Felix Christian Clemen for pointing out that the argument in \cite{Mubayicode} can be adapted to show that the limit defining $\gamma^+(\mathcal{F})$ exists. We also thank Balazs Patk\'os for many useful discussions.

\bibliographystyle{abbrvurl}

\end{document}